\documentclass[final]{paper}
\usepackage{amsfonts,latexsym,url,amsmath,amssymb,comment,enumitem}

\excludecomment{commentforus}

\usepackage{url,centernot}
\usepackage{graphicx,color}
\usepackage[color,notref,notcite]{showkeys}
\usepackage{bm}

\newcommand{\mathbi}[1]{{\boldsymbol #1}}
\def\N{\mathbf{N}}

\def\R{\mathbf{R}}

\def\reg{\mathop{\rm reg}}
\def\regLLE{\reg\nolimits_{\textsc{LLE}}}
\def\regbc{\reg\nolimits_{\textsc{Ba}}}
\def\dist{{\rm dist}}

\def\tpen{R}
\def\iso{{\mathcal L}}
\def\emb{{\bm{\Phi}}}

\def\n{\mathbf{n}}
\def\nv{\mathbf{N}}
\def\ml{{\mbox{\tiny\textsc{ML}}}}

\def\P{\mathbb{P}}
\def\<{\langle}
\def\>{\rangle}

\def\partition{\mathcal U}
\def\gr{{\mathcal G}}

\def\dsp{\displaystyle} 
\def\div{{\rm div}}

\def\incr{Q}

\def\CSTcontrol{C_{\rm ctrl}}

\newcounter{cst}
\def \ctel#1{C_{\refstepcounter{cst}\label{#1}\thecst}}
\def \cter#1{C_{\ref{#1}}}

\def\bt{\begin{theorem}}
\def\et{\end{theorem}}
\def\bp{\begin{proposition}}
\def\ep{\end{proposition}}
\def\bl{\begin{lemma}}
\def\el{\end{lemma}}
\def\bc{\begin{corollary}}
\def\ec{\end{corollary}}
\def\bd{\begin{definition}}
\def\ed{\end{definition}}
\def\br{\begin{remark}}
\def\er{\end{remark}}

\def\cv{K}
\newcommand{\polyd}{{\mathcal T}}

\newcommand{\mesh}{{\mathcal M}}
\def\sizemesh{{h_\mesh}}
\newcommand{\edge}{{\sigma}}
\newcommand{\medge}{{|\sigma|}}
\newcommand{\edges}{{\mathcal E}}              % ensemble des aretes
\newcommand{\edgescv}{{{\edges}_\cv}}  
  
\newcommand{\edgesext}{{{\edges}_{\rm ext}}} % ensemble des aretes ext
\newcommand{\edgesint}{{{\edges}_{\rm int}}} % ensemble des aretes int

\newcommand{\centers}{\mathcal{P}}
\newcommand{\x}{\mathbi{x}}

\newcommand{\y}{\mathbi{y}}

\newcommand{\xcv}{{\mathbi{x}}_\cv}

\newcommand{\centeredge}{\overline{\mathbi{x}}_\edge} %centre de gravite de l'interface

\newcommand{\vertex}{{\mathsf{v}}}
\newcommand{\vertices}{{\mathcal V}}
\newcommand{\verticescv}{\vertices_K}

\newcommand{\bu}{\overline{u}}
\newcommand{\bvarphi}{\mathbi{\varphi}}

\newcommand{\be}{\begin{equation}}
\newcommand{\ee}{\end{equation}}
\newcommand{\disc}{{\mathcal D}}
\newcommand{\Ii}{{I_{\O}}}
\newcommand{\Ib}{{I_{\partial\O}}}

\renewcommand{\O}{{\Omega}}
\def\dualcell{{\mathcal K}}
\def\dr{\partial}
\def\bfn{\mathbf{n}}

\def\ncvedge{\n_{\cv,\edge}}

\renewcommand{\d}{{\rm d}}

\newcommand{\ba}{\begin{array}{llll}   }
\newcommand{\bac}{\begin{array}{c}}
\newcommand{\bari}{\begin{array}{r}}
\newcommand{\ea}{\end{array}}

\def\bcdisc{{\disc^{\mbox{\tiny\textsc{Ba}}}}}
\def\bcdiscm{{\disc_m^{\mbox{\tiny\textsc{Ba}}}}}
\def\obcdisc{{\overline{\disc}^{\mbox{\tiny\textsc{Ba}}}}}
\def\bcI{I^{\mbox{\tiny\textsc{Ba}}}}
\def\bcIU{I_U^{\mbox{\tiny\textsc{Ba}}}}

\newcommand{\RTk}{\mathbb{RT}_k}

\newcommand{\bv}{\boldsymbol v}
\newcommand{\bfw}{\boldsymbol w}
\newcommand{\bfV}{\boldsymbol{V}}
\newcommand{\bw}{\boldsymbol w}

\def\grad{\nabla}
\def\mcv{|\cv|}
\def\dcvedge{d_{\cv,\edge}}

\newtheorem{theorem}{Theorem}[section]
\newtheorem{remark}[theorem]{Remark}

\newtheorem{lemma}[theorem]{Lemma} 
\newtheorem{definition}[theorem]{Definition}
\newtheorem{proposition}[theorem]{Proposition}
\newtheorem{corollary}[theorem]{Corollary}

\numberwithin{equation}{section}

\usepackage[normalem]{ulem}
\definecolor{violet}{rgb}{0.580,0.,0.827}

\newcounter{cexp}
\catcode`\@=11
\def\terml#1{T_{\refstepcounter{cexp}\@bsphack
\protected@write\@auxout{}%
           {\string\newlabel{#1}{{\thecexp}{\thepage}}}\thecexp}}
\catcode`\@=12

\newenvironment{proof}{\noindent {\bf Proof} }{$\square$ }

\newenvironment{acknowledgement}{\noindent {\bf Acknowledgement} }{}

\begin{document}

\title{Gradient schemes: generic tools \\ for the numerical analysis of diffusion equations}
 
\author{
J\'erome Droniou$^1$, Robert Eymard$^2$ and Rapha\`ele Herbin$^3$\\
\small{$^1$School of Mathematical Sciences, Monash University}
\\
\small{$^2$Laboratoire d'Analyse et de Math\'ematiques Appliqu\'ees,   Universit\'e Paris-Est,
UMR 8050}
\\
\small{$^3$Laboratoire d'Analyse Topologie et Probabilit\'es, 
UMR 6632,  Universit\'e d'Aix-Marseille}
}

\date{\today}

\maketitle

\begin{abstract}
The gradient scheme framework is based on a small number of properties and encompasses a large number of numerical methods for diffusion models. 
We recall these properties and  develop some new generic tools associated with the gradient scheme framework. These tools enable us to prove that classical schemes are indeed gradient schemes, and allow us to perform a complete and generic study of the well-known (but rarely
well-studied) mass lumping process.
They also allow an easy check of the mathematical properties of new schemes, by
developing a generic process for eliminating unknowns via barycentric condensation, and by designing a concept of discrete functional analysis toolbox for schemes based on polytopal meshes.
\end{abstract}



\date{\today}

\maketitle

\section{Introduction}\label{sec-intro}

A wide variety of schemes have been developed in the last few years for the numerical simulation of anisotropic diffusion equations  on general meshes, see \cite{review,  mfdrev, 3Dbench} and references therein.
The rigorous analysis of these methods is crucial to ensure their robustness and convergence, and to avoid the pitfalls of methods seemingly well-defined but not converging to the proper model  \cite[Chapter III, \S 3.2]{phdfaille}.
The necessity to conduct this analysis for each method and each model has given rise to a number of general ideas which are re-used from one study to the other; a set of rather informal techniques has thus emerged over the years.
 
It is tempting to push further this idea of ``set of informal similar techniques'', to try and make it a formal mathematical theory. 
This boils down to finding common factors in the studies for all pairs (\emph{method},\emph{model}), and to extract the core properties that ensure the stability and convergence of numerical methods for a variety of models. 
Identifying these core properties greatly reduces the work, which then amounts to two tasks:
\begin{itemize}[leftmargin=6em]
\item[Task (1):] establish that a given numerical method satisfies the said properties;
\item[Task (2):] prove that these properties ensure the convergence of a method for all considered models.
\end{itemize}
Thus, the number of convergence studies is reduced from 
[Card(\emph{methods})$\times$ Card(\emph{models})] -- which corresponds to one per pair (\emph{method},\emph{model}) -- to [Card(\emph{methods}) +  Card(\emph{models})].
Card(\emph{methods}) studies are needed to prove
that each method satisfies the core properties, and Card(\emph{models}) studies are required
to prove
that an abstract method that satisfies the core properties is convergent for each model.

% (see Figure \ref{fig:combi2}).
%\begin{center}
%\begin{figure}
%\input{fig-combi2.pdf_t}
%\caption{\label{fig:combi2}Numerical methods and models. Each line represents
%one analysis: to prove that a given method fits into the unification framework
%(lines on the left), and to analyse the framework when applied to a given
%model (lines on the right).}
%\end{figure}
%\end{center}

Attempts at designing rigorous theories of unified convergence analysis for families of numerical methods are not new, see e.g. \cite{ciarlet1978finite,  dipietro2012mathematical, arn2010fin, EB14} for finite element, discontinuous Galerkin methods and compatible discretisation operators.
Recently, the gradient scheme framework was developed \cite{eym-12-sma,dro-12-gra}.
Not only does this framework provide a unifying framework for a number of methods (Task 1) -- conforming and non-conforming finite elements, finite volumes, mimetic finite differences, \ldots --   but it also enables complete convergence analyses for a wide variety of models of 2nd order diffusion PDEs (Task 2) -- linear, non-linear, non-local, degenerate, etc.  \cite{dro-12-gra, zamm2013,eym-12-stef,dro-14-sto,dro-14-deg, aln-14-vi,koala,BGGLM15,eymard2015appl, cances:hal-01119735} -- through the verification of a very small number of  properties (3 for linear models, 4 or 5 for non-linear models).

The purpose of this article is to bring gradient schemes one
step further towards a unification theory. 
Indeed, we develop a set of generic tools
that make Task (1) extremely simple for a great variety of
methods. In other words, using these tools we can produce short but complete
proofs that several numerical methods for 2nd order diffusion problems are
gradient schemes.

The paper is organised as follows. In Section \ref{sec:prop}, we present the gradient scheme framework. 
This framework is based on the notion of gradient discretisation, which defines discrete spaces and operators,  and on
five core properties, presented in Subsection \ref{sec:deftools}: coercivity, consistency, limit-conformity, compactness, and piecewise constant reconstruction. 
A gradient scheme is a gradient discretisation applied to a given diffusion model, consisting in a set of second order partial differential equations and boundary conditions.
Depending on the considered model, a gradient discretisation must satisfy
three, four or five of these core properties to give rise to a convergent gradient scheme.
In sections following \ref{sec:deftools}, we develop generic notions that are useful
to establish that particular methods fit into the framework.
More precisely, in Subsection \ref{sec:LL} we introduce the concept of local
linearly exact gradients, and we show that it implies one of the core properties -- the consistency of gradient discretisations. Subsection \ref{sec:bc} deals with the barycentric condensation of gradient discretisations,
which is a classical way to eliminate degrees of freedom. 
The gradient scheme framework enables us, in Subsection \ref{sec:masslump}, to rigorously
define the well-known technique of mass lumping, and to show that this technique does not affect 
the convergence of a given scheme.
In Subsection \ref{sec:dfa} we provide an analysis toolbox for
schemes based on polytopal meshes, and we introduce the novel notion of control of a gradient discretisation
by this toolbox. This notion enables us to establish three of the main properties (coercivity,
limit-conformity and compactness) and therefore completes the notion of local linearly 
exact gradient discretisations.

In Section \ref{sec:review}, we show that all methods in the following list
are gradient schemes and satisfy four of the five core properties (coercivity, consistency, limit-conformity, compactness):
conforming and non conforming finite elements, $\RTk$ mixed finite elements, multi-point flux approximation MPFA-O schemes, discrete duality finite volume (DDFV) schemes, hybrid mimetic mixed methods (HMM), nodal mimetic finite difference (nMFD) methods, vertex average gradient (VAG) methods. 
For these methods, the fifth property (piecewise constant reconstruction) is either satisfied by definition, or can be satisfied by a mass-lumped version in the sense of  Subsection \ref{sec:masslump}. The mass-lumped versions are only detailed in the important cases of the conforming and non-conforming $\P_1$ finite elements.
We show that the notions of local linearly exact gradient discretisations, and of control by polytopal toolboxes, apply to most of the considered methods,
and therefore provide very quick proofs that these methods satisfy the consistency, coercivity,
limit-conformity and compactness properties. Some of the schemes have already been more or less formally shown to be gradient schemes in \cite{eym-12-sma,dro-12-gra}, but the proofs provided here thanks to the new
generic tools developed in Section \ref{sec:prop} are much more efficient and elegant than 
in previous works, and can be easily extended to other schemes.

A short conclusion is provided in Section \ref{sec:ccl}.

\section{Gradient discretisations: definitions and analysis tools}\label{sec:prop}

For simplicity we restrict ourselves to homogeneous Dirichlet boundary conditions; all other 
classical boundary conditions (non-homogeneous Dirichlet, Neumann, Fourier or mixed) can
be dealt with seamlessly in the gradient schemes framework \cite{koala}.
The principle of gradient schemes is to write the weak formulation of the PDE by replacing all continuous spaces and operators by discrete analogs. 
These discrete objects are described in a gradient discretisation.
Once a gradient discretisation is defined, its application to a given problem then leads to a gradient scheme.

For linear models, the convergence of gradient schemes
is obtained via error estimates based on the \emph{consistency} and
\emph{limit-conformity} measures $S_\disc$ and $W_\disc$.
For non-linear models, whose solutions may lack regularity or even be non-unique, error estimates may not always be obtained; however, convergence of approximate solutions can be obtained via compactness techniques such as those developed in the finite volume framework \cite{EGH00,eym-07-ana,review}. Even though they do not yield an explicit rate of convergence, these compactness techniques provide strong convergence results  -- such as uniform-in-time convergence \cite{dro-14-deg} -- under assumptions
that are compatible with field applications (discontinuous data, fully non-linear
models, etc.).

\subsection{Definitions} \label{sec:deftools}

\begin{definition}[Gradient discretisation for homogeneous Dirichlet boundary conditions]
\label{defgraddisc} Let $p\in (1,\infty)$ and let $\Omega$ be a bounded open subset of $\R^d$, where $d\in\N\setminus\{0\}$ is the space dimension.
The triplet $\disc = (X_{\disc,0},\Pi_\disc,\nabla_\disc)$ is a gradient discretisation for problems posed on $\Omega$ with homogeneous Dirichlet boundary conditions if:

\begin{enumerate}
\item $X_{\disc,0}$ is a finite dimensional space encoding the degrees of freedom (and accounting for the homogeneous Dirichlet boundary conditions),
\item $\Pi_\disc~:~X_{\disc,0}\to L^p(\O)$ is a linear mapping reconstructing a function in $L^p(\Omega)$ from the degrees of freedom,
\item $\nabla_\disc~:~X_{\disc,0}\to L^p(\O)^d$ is a linear mapping defining a discrete gradient from the degrees of freedom,
\item    $\Vert \nabla_\disc \cdot \Vert_{L^p(\O)^d}$ is a norm on $X_{\disc,0}$.
\end{enumerate}
\end{definition}

Here are the three properties a gradient discretisation needs to satisfy to enable the analysis of the corresponding gradient scheme on linear problems:
\begin{itemize}
\item The \emph{coercivity} ensures uniform discrete Poincar\'e inequalities
for the family of gradient discretisations; this is essential to obtain \emph{a priori}
estimates on the solutions to gradient schemes.
\item The \emph{consistency} states that the family of gradient discretisations
``covers'' the whole energy space of the model (e.g. $H^1_0(\O)$ for the linear
equation \eqref{pblin}).
\item The \emph{limit-conformity} ensures that the family of gradient and function reconstructions asymptotically satisfies the Stokes formula.
\end{itemize}

\begin{definition}[Coercivity] \label{def-coer}
Let $\disc$ be a gradient discretisation in the sense of Definition \ref{defgraddisc} and let $C_\disc$ be the norm of the linear mapping $\Pi_\disc$ defined by
\[
C_\disc =  \max_{u\in X_{\disc,0}\setminus\{0\}}\frac {\Vert \Pi_\disc u\Vert_{L^p(\O)}} {\Vert \nabla_\disc u \Vert_{L^p(\O)^d}}.
\]
A sequence $(\disc_m)_{m\in\N}$ of gradient discretisations in the sense of Definition \ref{defgraddisc} is said to be coercive if  there exists $ C_P \in \R_+$ such that $C_{\disc_m} \le C_P$ for all $m\in\N$.
\end{definition}

\begin{definition}[Consistency] \label{def-cons}
Let $\disc$ be a gradient discretisation in the sense of Definition \ref{defgraddisc} and
let $S_{\disc}:W^{1,p}_0(\O)\to [0,+\infty)$ be defined by
\[
\forall \varphi\in W^{1,p}_0(\O)\,,\;
\dsp S_{\disc}(\varphi) = \min_{u\in X_{\disc,0}}\left(\Vert \Pi_\disc u - \varphi\Vert_{L^p(\O)} + \Vert \nabla_\disc u -
\nabla\varphi\Vert_{L^p(\O)^d}\right).
\]
A sequence $(\disc_m)_{m\in\N}$ of gradient discretisations in the sense of Definition \ref{defgraddisc} is said to be consistent if for all $\varphi\in W^{1,p}_0(\O)$
we have $\lim_{m\to\infty} S_{\disc_m}(\varphi)=0$.
\end{definition}

\begin{definition}[Limit-conformity] \label{def-limconf}
Let $\disc$ be a gradient discretisation in the sense of Definition \ref{defgraddisc}.
We set $p'=\frac{p}{p-1}$, the dual exponent of $p$, and
$W^{\div, p'}(\O) = \{\bvarphi\in L^{p'}(\O)^d\,,\; \div\bvarphi\in L^{p'}(\O)\}$,
and we define
\[
\forall \bvarphi\in W^{\div,p'}(\O)\,,\;
\dsp W_{\disc}(\bvarphi) = \sup_{u\in X_{\disc,0}\setminus\{0\}}\frac{1}{\Vert  \nabla_\disc u \Vert_{L^p(\O)^d}}\left|\int_\O \left(\grad_\disc u(\x)\cdot\bvarphi(\x) + \Pi_\disc u(\x) \div\bvarphi(\x)\right)  \d\x
\right|.
\]
A sequence $(\disc_m)_{m\in\N}$ of gradient discretisations is said to be
limit-conforming if for all $\bvarphi\in W^{\div,p'}(\O)$
we have $\lim_{m\to\infty} W_{\disc_m}(\bvarphi) = 0$.
\end{definition}

To give an idea of how a gradient discretisation gives a converging gradient scheme 
for diffusion equations, let us consider the case of a linear elliptic equation
\begin{equation}\label{pblin}
\left\{\begin{array}{ll}
-\div(A\nabla\bu)=f&\mbox{ in $\O$},\\
\bu=0&\mbox{ on $\partial\O$},
\end{array}\right.\end{equation}
where $A:\O\mapsto\mathcal M_d(\R)$
is a measurable bounded and uniformly elliptic matrix-valued function such
that $A(\x)$ is symmetric for a.e. $\x\in \O$, and
$f\in L^2(\O)$. The solution to problem \eqref{pblin} is understood
in the weak sense:
\begin{equation}\label{pblinw}
\mbox{Find $\bu\in H^1_0(\O)$ such that, for all $\overline{v}\in H^1_0(\O)$,} \quad
\int_\O A(\x)\nabla \bu(\x)\cdot\nabla \overline{v}(\x) \d\x=\int_\O f(\x)\overline{v}(\x)\d\x.
\end{equation}
If $\disc$ is a gradient discretisation with $p=2$, then the corresponding gradient scheme
for \eqref{pblin} consists in writing
\begin{equation}\label{gslin}
\mbox{Find $u\in X_{\disc,0}$ such that, for all $v\in X_{\disc,0}$,}
\quad \int_\O A(\x)\nabla_\disc u(\x)\cdot\nabla_\disc v(\x)\d\x=\int_\O f(\x)
\Pi_\disc v(\x)\d\x.
\end{equation}
As seen here, \eqref{gslin} consists in replacing in \eqref{pblin}
the continuous space $H^1_0(\O)$ and the continuous gradient and function
by their discrete reconstruction from $\disc$.
Reference \cite{eym-12-sma} proves the following error estimate between the solution to \eqref{pblinw} and its gradient scheme approximation \eqref{gslin}:
\be\label{errlin}
\Vert \nabla \bu - \nabla_\disc u\Vert_{L^2(\O)^d}
+ \Vert \bu -\Pi_\disc  u\Vert_{L^2(\O)}
\le   \cter{Cerlin}\left[W_\disc(A\grad \bu) + S_\disc(\bu)\right],
\ee
where $\ctel{Cerlin}$ only depends on $A$ and an upper bound of $C_\disc$ in Definition
\ref{def-coer}. 
This shows that if a sequence $(\disc_m)_{m\in\N}$ of gradient discretisations
is coercive, consistent and limit-conforming, and if $(u_m)_{m\in\N}$ is a sequence of solutions to the corresponding gradient schemes for \eqref{pblin},
then $\Pi_{\disc_m}u_m\to \bu$ in $L^2(\O)$
and $\nabla_{\disc_m}u_m\to\nabla\bu$ in $L^2(\O)^d$. 
The study of a scheme for \eqref{pblin} then amounts to finding a gradient discretisation $\disc$ such that the scheme can be written under the form \eqref{gslin}, and to proving that sequences of such gradient discretisations satisfy the above described properties.
Establishing the consistency and limit-conformity usually consists in obtaining
estimates on $S_\disc$ and $W_\disc$ that give explicit rates of convergence in \eqref{errlin}.

Dealing with non-linear problems might additionally require
one or both of the following properties.
\begin{itemize}
\item The \emph{compactness} is used to deal with low-order non-linearities --
e.g. in semi- or quasi-linear equations.
\item The \emph{piecewise constant reconstruction} corresponds to mass-lumping
and is required to manage certain monotone non-linearities, or non-linearities on the time 
derivative as in Richards' model.
\end{itemize}

\begin{definition}[Compactness] \label{def-comp}
A sequence $(\disc_m)_{m\in\N}$ of gradient discretisations in the sense of Definition \ref{defgraddisc}
is said to be compact if, for any sequence $(u_m)_{m\in \N}$ such that $u_m\in X_{\disc_m,0}$ for $m \in \N$  and  $(\Vert  \nabla_{\disc_m}u_m \Vert_{L^p(\O)^d})_{m\in\N}$ is bounded, the sequence $(\Pi_{\disc_m}u_m)_{m\in\N}$  is relatively compact in $L^p(\O)$.
\end{definition}

\begin{definition}[Piecewise constant reconstruction] \label{def:pwrec}
Let  $\disc = (X_{\disc,0}, \Pi_\disc,\nabla_\disc)$ be a gradient discretisation in the sense of Definition \ref{defgraddisc}.
The linear mapping  $\Pi_\disc~:~X_{\disc,0}\to L^p(\Omega)$ is a piecewise constant reconstruction if there exists a finite set $B$, a basis $(e_i)_{i\in B}$ of $X_{\disc,0}$ and a family of (possibly empty) disjoint subsets $(V_i)_{i\in B}$ of $\O$ such that for all $u=\sum_{i\in B}u_i e_i\in X_{\disc,0}$
we have $\Pi_\disc u = \sum_{i\in B}u_i\chi_{V_i}$, where $\chi_{V_i}$
is the characteristic function of $V_i$.
\end{definition}

\begin{remark}\label{rem:setiforpiecewise}
Piecewise constant reconstructions generally use as set $B$ the set $I$ of geometrical entities attached to the degrees of freedom of the method (see Definition \ref{def:LL}); 
in this case  $(e_i)_{i\in B}$ is the canonical basis of $X_{\disc,0}$. 
Note that it is possible, starting from
a generic gradient discretisation $\disc$, to replace the original reconstruction $\Pi_\disc$ by
a reconstruction that is piecewise constant; the new gradient discretisation thus obtained is called a mass-lumped version of $\disc$ (see Section \ref{sec:masslump}).
\end{remark}

As an illustration of the use of  the importance of these properties for nonlinear problems, let us consider the following semi-linear modification of \eqref{pblin}:
\begin{equation}\label{pbnl}
\left\{\begin{array}{ll}
-\div(A\nabla\bu)+\beta(\bu)=f&\mbox{ in $\O$},\\
\bu=0&\mbox{ on $\partial\O$},
\end{array}\right.\end{equation}
for some function $\beta$ such that $\beta(s)s\ge 0$ for all $s\in\R$.
The gradient discretisation of this problem is pretty straightforward:
find $u\in X_{\disc,0}$ such that, for all $v\in X_{\disc,0}$,
\be\label{gsnl}
\int_\O A(\x)\nabla_\disc u(\x)\cdot\nabla_\disc v(\x)\d\x
+\int_\O \beta(\Pi_\disc u(\x))\Pi_\disc v(\x)\d\x=\int_\O f(\x)\Pi_\disc v(\x)\d\x.
\ee
The compactness property implies that an estimate on a sequence of discrete gradient $(\nabla_{\disc_m} u_m)_{m\in \N}$ will yield relative compactness of the corresponding sequence of reconstructions $(\Pi_{\disc_m} u_m)_{m\in \N}$, thus enabling a passing to the limit in the nonlinearity $\beta$. 

The formulation \eqref{gsnl} ensures \emph{a priori} estimates on the solution since,
taking $v=u$, the term $\beta(\Pi_\disc u)\Pi_\disc u$ is non-negative. However,
in practical implementation, computing $\int_\O \beta(\Pi_\disc u(\x))\Pi_\disc v(\x)\d\x$
might be problematic; even if $\Pi_\disc u$ and $\Pi_\disc v$ are polynomials
on each cell of a mesh (as in finite element schemes), $\beta(\Pi_\disc u)\Pi_\disc v$ is locally
not polynomial and no exact quadrature rule might exist to compute its integral.
An alternative scheme consists in replacing, in \eqref{gsnl}, the term
\be\label{gsnlA}
\int_\O \beta(\Pi_\disc u(\x))\Pi_\disc v(\x)\d\x
\ee
with
\be\label{gsnlB}
\int_\O \Pi_\disc \beta(u)(\x)\Pi_\disc v(\x)\d\x,
\ee
where $\beta(u) \in X_{\disc,0}$ is defined degree-of-freedom per degree-of-freedom,
that is $(\beta(u))_i=\beta(u_i)$ for all $i\in B$ with the notations of Definition \ref{def:pwrec}.
For finite element methods, \eqref{gsnlB} consists in integrating polynomials on
cells, and exact quadrature rules can be used. However, this alternative scheme does not ensure
\emph{a priori} estimates on the solution since, with the choice $v=u$, the term
$\Pi_\disc \beta(u)\Pi_\disc u$ might be negative in some part of the domain. Hence,
we have to choose between unconditional stability (\emph{a priori} estimates) with \eqref{gsnlA},
or a scheme that is practical to implement with \eqref{gsnlB}.

One of the interests of piecewise constant reconstructions is to solve this apparent
contradiction. If $\Pi_\disc$ is a piecewise constant reconstruction then 
\begin{equation}
 \Pi_\disc (\beta(u))=\beta(\Pi_\disc u).
\label{def:proppc}\end{equation}
Hence \eqref{gsnlA} and \eqref{gsnlB} are identical, and both
stability and computational practicality are satisfied. 
The second interest of piecewise constant resconstructions can be found in the analysis of
time-dependent problems. The discretisation of $\partial_t u$ leads to a term of the form
\be\label{time}
\int_\O \frac{\Pi_\disc u^{n+1}(\x)-\Pi_\disc u^n(\x)}{\delta\!t}\Pi_\disc v(\x)\d \x.
\ee
If $\Pi_\disc$ is a piecewise constant reconstruction, the mass matrix multiplying the coordinates $(u^{n+1}_i)_{i\in B}$ of $u^{n+1}$ in \eqref{time} is diagonal, and its inversion is therefore trivial.

As shown in \cite{dro-14-deg,eym-12-stef,zamm2013},
piecewise constant reconstructions ensure the stability and convergence
of gradient schemes for a variety of non-linear elliptic and parabolic equations.

\begin{remark}\label{rem:eqdef}
\begin{enumerate}
\item The consistency, limit-conformity and compactness of gradient discretisations may be defined in other equivalent ways \cite{koala}.
Moreover, the consistency of a sequence of gradient discretisations only needs to be checked
for $\varphi$ in a dense set of the domain of $S_\disc$
(e.g. $C^\infty_c(\O)$). The limit-conformity 
of a coercive sequence of gradient discretisations only needs to be checked for $\bvarphi$ in a dense set of the domain of $W_\disc$
(e.g. $C^\infty_c(\R^d)^d$, which is indeed dense in $W^{\div,p'}(\O)$ when $\O$ is locally star-shaped, which is
the case if $\O$ is polytopal).
Finally, the compactness of a sequence of gradient discretisations implies its coercivity.

\item Gradient discretisations for time-dependent problems can be easily deduced 
from the gradient discretisations for steady-state problems \cite{dro-12-gra,koala}.

\end{enumerate}
\end{remark}

 \subsection{Local linearly exact gradients}\label{sec:LL}

Most numerical methods for diffusion equations are based, explicitly or
implicitly, on local linearly exact reconstructed gradients. The following definition
gives a precise meaning to this.

\begin{definition}[Linearly exact gradient reconstructions]\label{def:linex}
Let $U$ be a bounded set of $\R^d$, let $I$ be a finite set and let $S = (\x_i)_{i\in I}$ be a family of points of $\R^d$.
A linear mapping $\gr:\R^I\mapsto L^\infty(U)^d$ is a linearly exact gradient reconstruction upon $S$ if, for any affine function $L:\R^d \to \R$, if $\xi=(L(\x_i))_{i\in I}$ then $\gr \xi = \nabla L$ on $U$.
The norm of $\gr$ is defined by 
\begin{equation}
 \Vert \gr\Vert_{\infty} = {\rm diam}(U)\max_{\xi\in \R^I\setminus\{0\}} \frac {||\gr \xi||_{L^\infty(U)^d}} {\max_{i\in I}|\xi_i|}.
\label{def:norlingrad}\end{equation}
\end{definition}

As expected, linearly exact gradient reconstructions enjoy nice approximation properties
when computed from interpolants of smooth functions.

\begin{lemma}[Estimate for linearly exact gradient reconstructions]\label{lem:stlinex}
Let $U$ be a bounded set of $\R^d$, let $S =(\x_i)_{i\in I}\subset \R^d$, and let
$\gr:\R^I\mapsto L^\infty(U)^d$ be a linearly exact gradient reconstruction upon $S$
in the sense of Definition \ref{def:linex}. Let $\varphi \in W^{2,\infty}(\R^d)$ and define $v\in \R^I$
by $v_i=\varphi(\x_i)$ for any $i\in I$. 
Then
\[
 |\gr v -\nabla\varphi|\le 
\left(1+ \frac 1 2\Vert \gr\Vert_{\infty}\left(\frac {\max_{i\in I} \dist(\x_i,U)}{ {\rm diam}(U)}  +1  \right)^2\right) {\rm diam}(U)||\varphi||_{W^{2,\infty}(\R^d)}
\quad\mbox{ a.e. on $U$}.
\]
\end{lemma}

\begin{proof} Take $\x_U\in U$ and let $L(\x)=\varphi(\x_U)+\nabla \varphi(\x_U)
\cdot(\x-\x_U)$ be the first order Taylor expansion of $\varphi$ around $\x_U$.
Let $\xi=(L(\x_i))_{i\in I}$. By linear exactness of $\gr$ we have
$\gr \xi=\nabla L=\nabla\varphi(\x_U)$ on $U$. Hence,
\be\label{est:lem:stlinex1}
|\gr\xi -\nabla\varphi|\le {\rm diam}(U)||\varphi||_{W^{2,\infty}(\R^d)}\quad\mbox{ on $U$}.
\ee
For any $i\in I$ we have $(v-\xi)_i=\varphi(\x_i)-L(\x_i)=\varphi(\x_i)
-\varphi(\x_U)-\nabla\varphi(\x_U) \cdot(\x_i-\x_U)$. Since $|\x_i-\x_U|\le \dist(\x_i,U) + {\rm diam}(U)$, we get
$|(v-\xi)_i|\le \frac 1 2 (\dist(\x_i,U) + {\rm diam}(U))^2||\varphi||_{W^{2,\infty}(\R^d)}$.
The linearity of $\gr$ and the definition of its norm therefore imply, for a.e. $\x\in U$,
\begin{eqnarray*}
|\gr v(\x) - \gr \xi(\x)|=|\gr (v-\xi)(\x)|\le \frac{\Vert \gr\Vert_{\infty}}{{\rm diam}(U)}\frac 1 2 \left(\max_{i\in I}\dist(\x_i,U) + {\rm diam}(U)\right)^2
||\varphi||_{W^{2,\infty}(\R^d)}\\
\le \frac 1 2\Vert \gr\Vert_{\infty}{\rm diam}(U) \left(\frac {\max_{i\in I} \dist(\x_i,U)}{ {\rm diam}(U)}  +1 \right)^2||\varphi||_{W^{2,\infty}(\R^d)}. 
\end{eqnarray*}
Combined with \eqref{est:lem:stlinex1}, this completes the proof of the lemma. \end{proof}

The consistency of gradient discretisations based on linearly exact
gradient reconstructions follows. Let us first give the
the definition of such gradient discretisations.

\begin{definition}[LLE gradient discretisation]\label{def:LL}
The triplet $\disc=(X_{\disc,0},\Pi_\disc,\nabla_\disc)$ is  an LLE (for ``local linearly exact'') gradient discretisation if there exists a finite partition $\partition$ of $\O$, 
a set $I$ of geometrical entities attached to the degrees of freedom (dof), a finite family of approximation points $S = (\x_i)_{i\in I}\subset \R^d$ and, for any $U \in \partition$, a subset $I_U \subset I$ such that:
\begin{enumerate}
\item $X_{\disc,0}=\R^{\Ii}\times \{0\}^{\Ib}$, where  the set $I$ is partitioned into $\Ii$ (interior geometrical entities attached to the dof) and ${\Ib}$ (boundary geometrical entities attached to the dof).
\item There exists a family $(\alpha_i)_{i\in I}$ such that, for all $i \in I$, $\alpha_i \in L^\infty(\Omega)$ and
\be\label{ll:piD}\begin{array}{lll}
\mbox{(a) $\forall i\in I$, $\forall U\in\partition$, if $i\notin I_U$ then
$\alpha_i=0$ on $U$},\\
\dsp \mbox{(b) for a.e. $\x \in \Omega$},\ \sum_{i\in I}\alpha_i(\x) = 1\mbox{ and }\forall v\in X_{\disc,0},\
\Pi_\disc v(\x) = \sum_{i\in I}\alpha_i(\x) v_i.
\end{array}\ee
\item There exists a family $(\gr_U)_{U\in\partition}$ such that, for all
$U\in\partition$, $\gr_U:\R^{I_U}\mapsto L^\infty(U)^d$ is a linearly exact gradient reconstruction upon $(\x_i)_{i\in I_U}$, in the sense of Definition \ref{def:linex},
and $\nabla_\disc v=\gr_U \bigl((v_i)_{i\in I_U}\bigr)$ on $U$, for all $v\in X_{\disc,0}$.

\item $\Vert \nabla_\disc \cdot \Vert_{L^p(\O)^d}$ is a norm on $X_{\disc,0}$.% (this implies that $I= \bigcup_{U \in \partition} I_U$).
\end{enumerate}
In that case, we define the LLE regularity of $\disc$ by
\begin{equation}
  \regLLE(\disc)=\max_{U\in\partition} \left(\Vert \gr_U\Vert_{\infty} + \max_{i\in I_U}
\frac{{\rm dist}(\x_i,U)}{{\rm diam}(U)}\right) +  \mathop{\rm esssup}\limits_{\x\in \Omega} \sum_{i\in I} \vert \alpha_i(\x) \vert.
\label{def:reglle}
\end{equation}
\end{definition}

\begin{remark}\label{rem:identicalpoints}
As implied by the terminology, an LLE gradient discretisation is a gradient discretisation in the sense of Definition \ref{defgraddisc}. 
Note that the existence of $i,j\in I$ with $i\neq j$ and $\x_i = \x_j$ is not excluded (see, e.g., Section \ref{sec:mpfa}).
\end{remark}

\begin{remark}\label{rem:estalpha}
We do not request $\Pi_\disc v$ to be linearly exact ($\alpha_i$ is not necessarily affine in each $U$); this reconstruction just needs to be computable from local degrees of freedom,
and exact on interpolants of constant functions.
This enables us to consider mass-lumped gradient discretisations.

In a number of cases, estimating $\sum_{i\in I}|\alpha_i(\x)|$ for a.e. $\x\in \Omega$ is trivial. 
For example, if for a.e. $\x\in \Omega$, there is exactly one $i\in I$ such that $\alpha_i(\x)=1$ and   $\alpha_j(\x)=0$ for all $j\in I\setminus\{i\}$, we get $\sum_{i\in I}|\alpha_i(\x)|= 1$ a.e. (then $\disc$ has a piecewise constant reconstruction and the set $B$ defined in Definition \ref{def:pwrec} is identical to $I$). Another example is the case where, for a.e. $\x\in U$, $\Pi_\disc v(\x)$ is a convex combination of the dof $(v_i)_{i\in I_U}$  (which is the case, e.g., if $\Pi_\disc v$ is linear on $U$, $v_i=\Pi_\disc v(\x_i)$ and $(\x_i)_{i\in I_U}$ are extremal points of $U$); then
$\alpha_i\ge 0$ for all $i\in I$ and $\sum_{i\in I}|\alpha_i(\x)|= 1$.
\end{remark}

\begin{proposition}[LLE gradient discretisations are consistent]\label{prop:LL}
Let $(\disc_m)_{m\in\N}$ be a sequence of LLE gradient discretisations,
associated for any $m\in\N$ to a partition $\partition_m$.
If $(\regLLE(\disc_m))_{m\in\N}$ is bounded and if $\displaystyle\max_{U\in\partition_m}{\rm diam}(U)$ $\to 0$ as $m\to\infty$, then $(\disc_m)_{m\in\N}$ is consistent in the sense of Definition
\ref{def-cons}.
\end{proposition}

\begin{proof} 
Let $\varphi\in C^\infty_c(\O)$ and let $v^m=(\varphi(\x^m_i))_{i\in I^m}\in X_{\disc_m,0}$,
where $S_m=(\x^m_i)_{i\in I^m}$ is the family of approximation points  of $\disc_m$.
Owing to Lemma \ref{lem:stlinex} we have, for $U\in \partition_m$ and
a.e. $\x\in U$,
\begin{multline}\label{cv:grad}
|\nabla_{\disc_m}v^m(\x)-\nabla\varphi(\x)|=
|\gr^m_U((v^m_i)_{i\in I^m_U})(\x)-\nabla\varphi(\x)| \\
\le \left(1+ \frac 1 2\regLLE(\disc_m)(\regLLE(\disc_m) +1 )^2\right)
{\rm diam}(U)||\varphi||_{W^{2,\infty}(\R^d)}.
\end{multline}
Let us now evaluate $|\Pi_{\disc_m}v^m-\varphi|$. Since any $(\x^m_i)_{i\in I^m_U}$ is within
distance $\regLLE(\disc_m){\rm diam}(U)$ of $U$, for all $i\in I^m_U$ and all $\x\in U$ we have
$|v^m_i-\varphi(\x)|\le (1+\regLLE(\disc_m)){\rm diam}(U)||\varphi||_{W^{1,\infty}(\R^d)}$. By \eqref{ll:piD}, we infer that for a.e. $\x\in U$
\begin{multline}\label{cv:pid}
|\Pi_{\disc_m}v^m(\x)-\varphi(\x)|=\left|\sum_{i\in I^m_U}\alpha^m_i(\x)(v_i^m-\varphi(\x))\right|
\le \sum_{i\in I^m_U}|\alpha^m_i(\x)|~\sup_{i\in I^m_U}|v^m_i-\varphi(\x)|\\
\le \regLLE(\disc_m)(1+\regLLE(\disc_m)){\rm diam}(U)||\varphi||_{W^{1,\infty}(\R^d)}.
\end{multline}
Estimates \eqref{cv:grad} and \eqref{cv:pid} and the
assumptions on $(\disc_m)_{m\in\N}$ show that $\nabla_{\disc_m}v^m
\to \nabla \varphi$ in $L^\infty(\O)^d$ and $\Pi_{\disc_m}v^m\to \varphi$ in $L^\infty(\O)$ as $m\to\infty$. Remark \ref{rem:eqdef} then concludes the proof. \end{proof}

%\begin{remark} The term ${\rm diam}(U)$ in $\regLLE(\disc)$ could be replaced with
%any quantity $\omega_U>0$, the requirement to prove Proposition \ref{prop:LL} being that, for a sequence $(\disc_m)_{m\in\N}$ of LLE gradient discretisations,
%associated for any $m\in\N$ to a partition $\partition_m$, there holds $\max_{U\in\partition_m}\omega_U\to 0$ as $m\to\infty$.
%\end{remark}

\subsection{Barycentric elimination of degrees of freedom}\label{sec:bc}

The construction of a given scheme often requires several  interpolation points.
However, some of these points can be eliminated afterwards to reduce the computational
cost. A classical way to perform this reduction of degrees of freedom is
through barycentric combinations, by replacing certain
unknowns with averages of other unknowns. We describe here
a way to perform this reduction in the general context of LLE gradient discretisations,
while preserving the required properties
(coercivity, consistency, limit-conformity and compactness).

\begin{definition}[Barycentric condensation of an LLE gradient discretisation]\label{def:bcgd}
Let $\disc$ be an LLE gradient discretisation. 
We denote by $S = (\x_i)_{i\in I}\subset \R^d$  the  family of approximation points of $\disc$ and by $\partition$ its partition. A gradient discretisation $\bcdisc$ is a barycentric condensation of ${\disc}$ 
if there exists $\bcI\subset I$ and, for all $i\in I\backslash \bcI$,
a set $H_i\subset \bcI$ and real numbers
$(\beta^i_{ j})_{{j}\in H_i}$ satisfying 
\be\label{barcomb}
\sum_{{j}\in H_i}\beta^i_{j} =1\quad\mbox{ and }\quad
\sum_{{j}\in H_i}\beta^i_{j} \x_{j}=\x_i,
\ee
such that
\begin{itemize}
\item ${\Ib}\subset {\bcI}$.
\item $X_{\bcdisc,0}=\R^{{\bcI}\cap \Ii}\times \{0\}^{\Ib}$.
\item For all $v\in X_{\bcdisc,0}$ we have $\Pi_{\bcdisc}v=\Pi_\disc V$
and $\nabla_{\bcdisc} v=\nabla_\disc V$, where $V\in X_{\disc,0}=
\R^{\Ii}\times \{0\}^{\Ib}$ is defined
by
\be\label{bc:extv}
\forall i\in I\,,\;V_i = \left\{\begin{array}{ll}v_i&\mbox{ if $i\in {\bcI}$},\\
\sum_{{j}\in H_i}\beta^i_{j}  v_{j}&\mbox{ if $i\in I\setminus{\bcI}$}.
\end{array}\right.
\ee
(We note that $V$ is indeed in $X_{\disc,0}$ since ${\Ib}\subset {\bcI}$ and $v_i=0$
if $i\in {\Ib}$.)
\end{itemize}
We define the regularity of the barycentric condensation $\bcdisc$ by
\[
\regbc(\bcdisc)=1+\max_{i\in I\setminus {\bcI}}
\left(\sum_{{j}\in H_i}|\beta^i_{j}| + \max_{U\in\partition\,|\,i\in I_U}
\max_{{j}\in H_i} \frac{{\rm dist}(\x_j ,\x_i)}{{\rm diam}(U)}\right).
\]
\end{definition}

It is clear that the above defined barycentric condensation $\bcdisc$ is a gradient discretisation. 
Indeed, if $\nabla_\bcdisc v=0$ on $\O$ then $\nabla_\disc V=0$ on $\O$ and thus $V_i=0$ for all $i\in S$ (since $\disc$ is a gradient discretisation and therefore $||\nabla_\disc\cdot||_{L^p(\O)^d}$ is a norm on $X_{\disc,0}$). 
This shows that $v_i=0$ for all $i\in {\bcI}$, and thus that $||\nabla_\bcdisc \cdot||_{L^p(\O)^d}$ is a norm on $X_{\bcdisc,0}$.

\begin{remark}[Localness of the barycentric elimination]
Bounding the last term in $\regbc(\bcdisc)$ consists in requiring that, if $i\in I\setminus {\bcI}$ is involved in the definition of
$\gr_U$, then any degree of freedom ${j}\in H_i$ used to eliminate the degree of freedom $i$ lies within distance $\mathcal O({\rm diam}(U))$ of $U$. 
This ensures that, after barycentric elimination, $\gr_U$ is still computed using only
degrees of freedom in a neighborhood of $U$.
\end{remark}

Barycentric elimination expresses some degrees of freedom by combinations that are linearly exact. 
As a consequence, the LLE property is preserved in the process, and the consistency of barycentric condensations of LLE gradient discretisations
is ensured by Proposition \ref{prop:LL}.

\begin{lemma}[Barycentric elimination preserves the LLE property]\label{lem:bcll}
Let $\disc$ be an LLE gradient discretisation in the sense of Definition \ref{def:linex}, and let $\bcdisc$ be a barycentric condensation of $\disc$. 
Then $\bcdisc$ is an LLE gradient discretisation on the same partition as $\disc$, and $\regLLE(\bcdisc)\le \regbc(\bcdisc)\regLLE(\disc)+
\regbc(\bcdisc)+\regLLE(\disc)$.
\end{lemma}

\begin{proof} Obviously, ${\bcI}=({\bcI}\cap \Ii)\sqcup {\Ib}$ forms the
geometrical entities attached to the dof of $\bcdisc$ since $X_{\bcdisc,0}=\R^{{\bcI}\cap \Ii}\times \{0\}^{\Ib}$.
Let $\partition$ be the partition corresponding to $\disc$, and let $U\in\partition$. 
Take $v\in X_{\bcdisc,0}$ and let $V\in X_{\disc,0}$ be defined
by \eqref{bc:extv}. We notice that, for any $U\in\partition$, the values $(V_i)_{i\in I_U}$
are computed in terms of $(v_i)_{i\in \bcIU}$ with
$\bcIU=(I_U\cap {\bcI}) \cup \bigcup_{i\in I_U\backslash {\bcI}}H_i$.

We have, for $\x\in U$,
\[
\Pi_\bcdisc v(\x)=\Pi_\disc V(\x)=\sum_{i\in I_U}\alpha_i(\x) V_i
=\sum_{i\in I_U\cap {\bcI}}\alpha_i(\x) v_i
+\sum_{i\in I_U\backslash {\bcI}}\alpha_i(\x)\sum_{{j}\in H_i}\beta_{j}^i  v_{j}
=\sum_{i\in \bcIU} \widetilde{\alpha}_{i}(\x) v_i
\]
with 
\[
\begin{array}{ll}
\dsp\widetilde{\alpha}_{i}(\x)=\alpha_{i}(\x)+\sum_{k\in I_U\backslash {\bcI}\,|\,i\in H_k} \beta_{i}^{k} \alpha_{k}(\x)&\mbox{ if $i\in I_U\cap {\bcI}$},\\
\dsp \widetilde{\alpha}_{i}(\x)=\sum_{k\in I_U\backslash {\bcI}\,|\,i\in H_k} \beta_{i}^{k} \alpha_{k}(\x)
&\mbox{ if $i\in \bcIU\backslash I_U$}.
\end{array}
\]
Thanks to \eqref{barcomb} and \eqref{ll:piD} we have
\begin{multline}
\sum_{i\in \bcIU}\widetilde{\alpha}_{i}(\x)=\sum_{i\in I_U\cap {\bcI}}\alpha_i(\x)
+\sum_{i\in \bcIU\;}\sum_{k\in I_U\backslash {\bcI}\,|\,i\in H_{k}} \beta_{i}^{k}
\alpha_{k}(\x)\\
=\sum_{i\in I_U\cap {\bcI}}\alpha_i(\x)
+\sum_{{k}\in I_U\backslash {\bcI}}\alpha_{k}(\x)\sum_{i\in H_{k}} \beta_{i}^{k}
=\sum_{i\in I_U\cap {\bcI}}\alpha_i(\x)
+\sum_{{k}\in I_U\backslash {\bcI}}\alpha_{k}(\x)
=1.
\label{estloin}
\end{multline}
Hence, $\Pi_\bcdisc v$ has the required form.
The gradient $(\nabla_\bcdisc v)_{|U}=\gr_U ((V_i)_{i\in I_U})$
only depends on $(v_i)_{i\in \bcIU}$ and can thus be written
$\widetilde{\gr}_U ((v_i)_{i\in  \bcIU})$.
By \eqref{barcomb} the reconstruction $v\mapsto V$ is linearly
exact, that is if $v$ interpolates the values of an affine mapping $L$ at
the points $(\x_i)_{i\in \bcIU}$ then $V$ interpolates the same mapping $L$ at the
points $(\x_i)_{i\in I_U}$. Hence, the linear exactness of $\gr_U$
gives the linear exactness of $\widetilde{\gr}_U$. This completes the proof that
$\bcdisc$ is an LLE gradient discretisation.

Let us now establish the upper bound on $\regLLE(\bcdisc)$. 
For all $i\in I_U\backslash {\bcI}$ we have 
$|V_i|\le \sum_{{j}\in H_i}|\beta^i_{j}|\,|v_{j}|\le \regbc(\bcdisc)
\max_{j\in \bcIU}|v_{j}|$. This also holds for $i\in I_U\cap {\bcI}$ since
$\regbc(\bcdisc)\ge 1$. Hence, a.e. on $U$,
\[
\left|\widetilde{\gr}_U \left((v_i)_{i\in \bcIU}\right)\right|=
\left|\gr_U \left((V_i)_{i\in I_U}\right)\right|\le 
\frac{\Vert \gr_U\Vert_{\infty}\regbc(\bcdisc)}{{\rm diam}(U)}\max_{i\in \bcIU}|v_i|
\]
and thus
\be\label{stabcons}
\mbox{$\Vert\widetilde{\gr}_U\Vert_{\infty} \le \Vert \gr_U\Vert_{\infty}\regbc(\bcdisc)$}.
\ee
Reproducing the reasoning in the first two equalities in \eqref{estloin}
with absolute values and inequalities, we see that
\be\label{estalpha}
\sum_{i\in \bcIU}|\widetilde{\alpha}_{i}(\x)|
\le
\sum_{i\in I_U\cap {\bcI}}|\alpha_i(\x)|
+\sum_{{k}\in I_U\backslash {\bcI}}|\alpha_{k}(\x)|\sum_{i\in H_{k}} |\beta_{i}^{k}|
\le \regbc(\bcdisc)\sum_{i\in I_U}|\alpha_i(\x)|.
\ee
Finally, for ${j}\in \bcIU$ we
estimate $\frac{{\rm dist}(\x_j ,U)}{{\rm diam}(U)}$ by studying
two cases. If ${j}\in I_U$ then ${\rm dist}(\x_j ,U)\le \regLLE(\disc){\rm diam}(U)$.
If ${j}\not\in I_U$ then there exists $i\in I_U\backslash {\bcI}$ such that
${j}\in H_i$, and thus
${\rm dist}(\x_j ,\x_i)\le \regbc(\bcdisc){\rm diam}(U)$; this
gives ${\rm dist}(\x_j ,U)\le (\regbc(\bcdisc)+\regLLE(\disc)){\rm diam}(U)$.
Combined with \eqref{stabcons} and \eqref{estalpha},
these estimates on ${\rm dist}(\x_j ,U)$ 
prove the bound on $\regLLE(\bcdisc)$ stated in the lemma. \end{proof}

Barycentric condensations of LLE gradient discretisations satisfy the same properties (coercivity, consistency, compactness, limit-conformity) as the original gradient discretisation.  
The coercivity, limit-conformity and compactness properties result from the fact that $X_{\bcdisc,0}$ is (roughly) a subspace of $X_{\disc,0}$, and the consistency is a consequence of Lemma \ref{lem:bcll} and Proposition
\ref{prop:LL}.

\begin{theorem}[Properties of barycentric condensations of gradient discretisations]\label{th:bcgd}
Let $(\disc_m)_{m\in\N}$ be a sequen\-ce of LLE gradient discretisations that is coercive, consistent,
limit-conforming and compact in the sense of the definitions in Section \ref{sec:deftools}.
Let $\partition_m$ be the finite partition of $\Omega$ corresponding to $\disc_m$. We assume that
$\max_{U\in\partition_m}{\rm diam}(U)\to 0$ as $m\to\infty$, and that
$(\regLLE(\disc_m))_{m\in\N}$ is bounded.
For any $m\in\N$ we take a barycentric condensation $\bcdiscm$ of
$\disc_m$ such that $(\regbc(\bcdiscm))_{m\in\N}$ is bounded.

Then $(\bcdiscm)_{m\in\N}$ is coercive, consistent, limit-conforming and compact.
\end{theorem}

\begin{proof} For any $v\in X_{\bcdiscm,0}$, with $V$ defined by \eqref{bc:extv} we have
\[
||\Pi_{\bcdiscm} v||_{L^p(\O)}=||\Pi_{\disc_m} V||_{L^p(\O)}
\le C_{\disc_m} ||\nabla_{\disc_m} V||_{L^p(\O)^d}=C_{\disc_m} ||\nabla_{\bcdiscm} v||_{L^p(\O)^d},
\]
which shows that $C_{\bcdiscm}\le C_{\disc_m}$ and thus that $(\bcdiscm)_{m\in\N}$
is coercive. To prove the compactness, we take $(\nabla_{\bcdiscm}v_m)_{m\in\N}=
(\nabla_{\disc_m}V_m)_{m\in\N}$ bounded in $L^p(\O)^d$, and we use the compactness of
$(\disc_m)_{m\in\N}$ to see that $(\Pi_{\disc_m}V_m)_{m\in\N}=(\Pi_{\bcdiscm}v_m)_{m\in\N}$
is relatively compact in $L^p(\O)$. The limit conformity follows by writing
\begin{multline*}
\frac{1}{||\nabla_\bcdiscm v||_{L^p(\O)^d}}
\left|\int_\O \left(\grad_\bcdiscm v(\x)\cdot\bvarphi(\x) + \Pi_\bcdiscm v(\x) \div\bvarphi(\x)\right)  \d\x\right|\\
=\frac{1}{||\nabla_{\disc_m} V||_{L^p(\O)^d}}
\left|\int_\O \left(\grad_{\disc_m} V(\x)\cdot\bvarphi(\x) + \Pi_{\disc_m} V(\x) \div\bvarphi(\x)\right)  \d\x\right|,
\end{multline*}
which shows that $W_{\bcdiscm}(\bvarphi)\le W_{\disc_m}(\bvarphi)$. Finally, 
by Lemma \ref{lem:bcll} each $\bcdiscm$ is an LLE gradient discretisation
and the boundedness of $(\regLLE(\disc_m))_{m\in\N}$ and $(\regbc(\bcdiscm))_{m\in\N}$
show that $(\regLLE(\bcdiscm))_{m\in\N}$ is bounded.
Proposition \ref{prop:LL} then gives the consistency of $(\bcdiscm)_{m\in\N}$. \end{proof}

\subsection{Mass lumping and comparison of reconstruction operators}\label{sec:masslump}

``Mass-lumping'' is the generic name of the process applied to modify schemes that do not have a built-in piecewise constant reconstruction, say for instance the $\P_1$ finite element scheme (see Section \ref{sec:galerkin}). 
This is often done on a case-by-case basis, with \emph{ad hoc} studies.
The gradient scheme framework provides an efficient generic setting for performing this mass-lumping.
The idea is to modify the reconstruction operator so that it  becomes a piecewise constant reconstruction; under an assumption that is easy to verify
in practice, this ``mass-lumped'' gradient discretisation can be compared
with the original gradient discretisation, which ensures that all properties
required for the convergence of the mass-lumped scheme are satisfied.

\begin{definition}[Mass-lumped gradient discretisation]\label{def:ml}
Let $\disc=(X_{\disc,0},\Pi_\disc,\nabla_\disc)$ be a gradient discretisation
in the sense of Definition \ref{defgraddisc}. A mass-lumped version of $\disc$ 
is a gradient discretisation $\disc^\ml=(X_{\disc,0},\Pi_\disc^\ml,\nabla_\disc)$ such that  $\Pi_\disc^\ml$ is a piecewise constant reconstruction in the sense of Definition 
\ref{def:pwrec}.
\end{definition}

% \begin{remark}[Mass lumping with respect to a non canonical basis]
% The basis $(e_i)_{i\in I}$ of $X_{\disc,0}$ is usually chosen in a canonical way, each vector in this basis corresponding to a natural degree of freedom of $\disc$ (see examples in the $\P_1$ and non-conforming $\P_1$ case in Section \ref{sec:galerkin}).
% Mass-lumping could be done with respect to a non-standard basis, but this
% might lead to additional numerical cost if the computation of $\nabla_\disc$
% in this non-standard basis is complex; the scheme  implementation might require to
% perform changes of basis, possibly with full transition matrices, to compute
% $\Pi_{\disc}^\ml$ and $\nabla_\disc$.
% \end{remark}
In all the cases of mass-lumping considered in this paper, we show that the following theorem applies to $\disc_m^\star = \disc_m^\ml$. This theorem states that, if two
sequences of gradient discretisations share the same space and reconstructed gradients,
one inherits the properties from the other provided that their reconstruction operators
are close to each other (condition \eqref{est:ml}). Moreover, it also establishes that the sufficient condition \eqref{est:ml} is also necessary for the mass-lumped schemes to  satisfy the compactness and limit-conformity properties, since these properties are satisfied by all the considered initial schemes.

\begin{theorem}[Comparison of reconstruction operators]\label{th:ml}
Let $(\disc_m)_{m\in\N}$ be a sequence of gradient discretisations in the sense
of Definition \ref{defgraddisc}.
For any $m\in\N$, let $\disc_m^\star$ be a gradient discretisation defined from $\disc_m$ by $\disc_m^\star = (X_{\disc_m,0},\Pi_{\disc_m}^\star,\grad_{\disc_m}) $, where $\Pi_{\disc_m}^\star$ is a linear operator from $X_{\disc_m,0}$ to $L^p(\Omega)$.
\begin{enumerate}
\item We assume that there exists a sequence $(\omega_m)_{m\in\N}$ such that
\be\label{est:ml}
\begin{array}{l}
\lim_{m\to\infty}\omega_m=0,\mbox{ and }\\
\dsp\forall m\in\N\,,\;\forall v\in X_{\disc_m,0}\,,\;||\Pi_{\disc_m}^\star v -\Pi_{\disc_m}v||_{L^p(\O)}\le \omega_m ||\nabla_{\disc_m}v||_{L^p(\O)^d}.
\end{array}
\ee
If $(\disc_m)_{m\in\N}$ is coercive (resp. consistent, limit-conforming, or
compact -- in the sense of the definitions in Section \ref{sec:deftools}), 
then $(\disc_m^\star)_{m\in\N}$ is
also coercive (resp. consistent, limit-conforming,
or compact).
\item
Reciprocally, if $(\disc_m)_{m\in\N}$ and  $(\disc_m^\star)_{m\in\N}$ are both compact and limit-conforming in  the sense of the definitions in Section \ref{sec:deftools}, then 
there exists $(\omega_m)_{m\in\N}$ such that \eqref{est:ml} holds.
\end{enumerate}
\end{theorem}

\begin{proof}
Let us prove the first item of the theorem.
We let $M = \sup_{m\in\N} \omega_m$, and we use the triangular inequality to
write, from \eqref{est:ml}, for any $v\in X_{\disc_m,0}$,
\be\label{comp:dg33}
||\Pi_{\disc_m}^\star v||_{L^p(\O)}\le
||\Pi_{\disc_m}^\star v-\Pi_{\disc_m}v||_{L^p(\O)}+||\Pi_{\disc_m}v||_{L^p(\O)}
\le M||\nabla_{\disc_m}v||_{L^p(\O)^d}+||\Pi_{\disc_m}v||_{L^p(\O)}.
\ee

\textsc{Coercivity}: let us assume that $(\disc_m)_{m\in\N}$ is coercive with constant $C_P$. Using \eqref{comp:dg33} we find
$||\Pi_{\disc_m}^\star v||_{L^p(\O)}\le (M+C_P)||\nabla_{\disc_m}v||_{L^p(\O)^d}$ and
the coercivity of $(\disc_m^\star)_{m\in\N}$ follows.

\textsc{Consistency}: let us assume that $(\disc_m)_{m\in\N}$ is consistent. 
Using the triangular inequality and \eqref{est:ml}, we write, for any $v\in X_{\disc_m,0}$ and  $\varphi\in W^{1,p}_0(\O)$,
\begin{align*}
 S_{\disc_m^\star}(\varphi) &\le ||\Pi_{\disc_m^\star }v-\varphi||_{L^p(\O)}
  +||\nabla_{\disc_m}v- \nabla\varphi||_{L^p(\O)^d}\\
    &\le 
  \omega_m||\nabla_{\disc_m}v||_{L^p(\O)^d}
  +||\Pi_{\disc_m} v-\varphi||_{L^p(\O)}
  +||\nabla_{\disc_m} v- \nabla\varphi||_{L^p(\O)^d}\\
  &\le \omega_m||\nabla\varphi||_{L^p(\O)^d}
+||\Pi_{\disc_m} v-\varphi||_{L^p(\O)}
+(1+\omega_m)||\nabla_{\disc_m} v- \nabla\varphi||_{L^p(\O)^d}\\
 &\le \omega_m||\nabla\varphi||_{L^p(\O)^d}
+(1+M)(||\Pi_{\disc_m} v-\varphi||_{L^p(\O)}+||\nabla_{\disc_m} v- \nabla\varphi||_{L^p(\O)^d}).
\end{align*}
Hence $S_{\disc_m^\star}(\varphi)\le \omega_m||\nabla\varphi||_{L^p(\O)^d}
+(1+M)S_{\disc_m}(\varphi)$ and the consistency of $(\disc_m^\star)_{m\in\N}$ follows
from the consistency of $(\disc_m)_{m\in\N}$ and from $\lim_{m\to\infty}\omega_m= 0$.

\textsc{Limit-conformity}: let us now assume that $(\disc_m)_{m\in\N}$ is limit-conforming. 
By the triangular inequality and \eqref{est:ml}, for any $\bvarphi\in W^{\div,p'}(\O)$,
\begin{align*}
\Bigg|\int_\O \Big(\grad_{\disc_m} &v(\x)\cdot\bvarphi(\x) + \Pi_{\disc_m}^\star v(\x) \div\bvarphi(\x)\Big)  \d\x\Bigg|\\
\le{}& ||\div\bvarphi||_{L^{p'}(\O)}\omega_m||\nabla_{\disc_m}v||_{L^p(\O)^d}
+\left|\int_\O \left(\grad_{\disc_m} v(\x)\cdot\bvarphi(\x) + \Pi_{\disc_m} v(\x) \div\bvarphi(\x)\right)  \d\x\right|.
\end{align*}
Using \eqref{est:ml}, we infer that
$W_{\disc_m^\star}(\bvarphi)\le \omega_m ||\div\bvarphi||_{{p'}(\O)^d} 
+W_{\disc_m}(\bvarphi)\to 0$ as $m\to\infty$,
and the limit conformity of $(\disc_m^\star)_{m\in\N}$ is established.

\textsc{Compactness}: we now assume that $(\disc_m)_{m\in\N}$ is compact.
If $(\nabla_{\disc_m}v_m)_{m\in\N}$ is bounded in $L^p(\O)^d$, then the compactness of $(\disc_m)_{m\in\N}$ ensures
that $(\Pi_{\disc_m}v_m)_{m\in\N}$ is relatively compact in $L^p(\O)$. Since
$||\Pi_{\disc_m}^\star v_m-\Pi_{\disc_m}v_m||_{L^p(\O)}\to 0$ as $m\to\infty$
by \eqref{est:ml}, we deduce that $(\Pi_{\disc_m}^\star v_m)_{m\in\N}$ is relatively compact in $L^p(\O)$.

\smallskip

Let us now turn to the proof, by way of contradiction, of the second item.
We therefore assume that  $(\disc_m)_{m\in\N}$ and  $(\disc_m^\star)_{m\in\N}$ are both compact and limit-conforming, and that
\be
\omega_m := \max_{v\in X_{\disc_m,0}\backslash\{0\}}\frac{\Vert \Pi_{\disc_m}v-\Pi_{\disc_m}^\star v\Vert_{L^p(\O)}}{\Vert \nabla_{\disc_m}v\Vert_{L^p(\O)^d}}\centernot\longrightarrow 0\mbox{ as $m\to\infty$}.
\label{def:omegam}\ee
Then we can find $\varepsilon_0 >0$, a subsequence of $(\disc_m,\disc_m^\star)_{m\in\N}$
(not denoted differently) and for each $m\in\N$ an element $v_m\in X_{\disc_m,0}\backslash\{0\}$
such that $||\Pi_{\disc_m}^\star v_m -\Pi_{\disc_m}v_m||_{L^p(\O)}\ge \varepsilon_0 ||\nabla_{\disc_m}v_m||_{L^p(\O)^d}$. Since $v_m\neq 0$, we can consider $\widetilde{v}_m=\frac{v_m}{||\nabla_{\disc_m}v_m||_{L^p(\O)^d}}$, which satisfies $||\nabla_{\disc_m} \widetilde v_m||_{L^p(\O)^d} = 1$ and
\be\label{gtr.eps0}
||\Pi_{\disc_m}^\star \widetilde v_m -\Pi_{\disc_m} \widetilde v_m||_{L^p(\O)}\ge  \varepsilon_0.
\ee
We extract another subsequence such that $\nabla_{\disc_m} \widetilde v_m$ weakly converges to some $G$ in $L^p(\O)^d$, and, using the compactness of $(\disc_m)_{m\in\N}$ and $(\disc_m^\star)_{m\in\N}$, $\Pi_{\disc_m}\widetilde v_m\to v$ in $ L^p(\O)$ and $\Pi_{\disc_m}^\star\widetilde v_m\to
v^\star$ in $L^p(\O)$. Passing to the limit in \eqref{gtr.eps0} we find $||v - v^\star||_{L^p(\O)}\ge \varepsilon_0 $. Extending the functions $\nabla_{\disc_m} \widetilde v_m$,  $\Pi_{\disc_m}\widetilde v_m$ and $\Pi_{\disc_m}^\star\widetilde v_m$ by $0$ outside $\Omega$, we see that, for any $\bvarphi\in W^{\div,p'}(\R^d)$,
\[
\left|\int_{\R^d} \left(\grad_{\disc_m} \widetilde v_m(\x)\cdot\bvarphi(\x) + \Pi_{\disc_m}^\star \widetilde v_m(\x) \div\bvarphi(\x)\right)  \d\x\right|\le W_{\disc_m^\star}(\bvarphi|_\Omega),
\]
and
\[
\left|\int_{\R^d} \left(\grad_{\disc_m} \widetilde v_m(\x)\cdot\bvarphi(\x) + \Pi_{\disc_m} \widetilde v_m(\x) \div\bvarphi(\x)\right)  \d\x\right|\le W_{\disc_m}(\bvarphi|_\Omega).
\]
By limit-conformity of both sequences of gradient discretisations, we can let $m\to\infty$ and
we find 
\[
\int_{\R^d} \left(G\cdot\bvarphi(\x) +  v^\star(\x) \div\bvarphi(\x)\right)  \d\x=
\int_{\R^d} \left(G\cdot\bvarphi(\x) +  v(\x) \div\bvarphi(\x)\right)  \d\x= 0.
\]
This proves that $v, v^\star\in W^{1,p}_0(\Omega)$ and that $G = \grad v = \grad v^\star$.
Poincar\'e's inequality then gives $v = v^\star$, which contradicts $||v - v^\star||_{L^p(\O)}\ge \varepsilon_0 $. Therefore the sequence $(\omega_m)_{m\in\N}$ defined by \eqref{def:omegam} satisfies \eqref{est:ml}.
\end{proof}

\subsection{Polytopal meshes and discrete functional analysis}\label{sec:dfa}

Although gradient discretisations are not limited to mesh-based methods (for example it is easy to include spectral methods in this framework), a large number of schemes
for \eqref{pblin} are built on meshes. 

\begin{definition}[Polytopal mesh]\label{def:polymesh}~
Let $\Omega$ be a bounded polytopal open subset of $\R^d$ ($d\ge 1$). 
A polytopal mesh of $\O$ is given by $\polyd = (\mesh,\edges,\centers,\vertices)$, where:
\begin{enumerate}
\item $\mesh$ is a finite family of non empty connected polytopal open disjoint subsets of $\O$ (the cells) such that $\overline{\O}= \dsp{\cup_{K \in \mesh} \overline{K}}$.
For any $K\in\mesh$, $\mcv>0$ is the measure of $K$ and $h_K$ denotes the diameter of $K$.

\item $\edges$ is a finite family of disjoint subsets of $\overline{\O}$ (the edges of the mesh in 2D,
the faces in 3D), such that any $\edge\in\edges$ is a non empty open subset of a hyperplane of $\R^d$ and $\edge\subset \overline{\O}$.
We assume that for all $K \in \mesh$ there exists  a subset $\edgescv$ of $\edges$
such that $\dr K  = \dsp{\cup_{\edge \in \edgescv}} \overline{\edge}$. 
We then denote by $\mesh_\edge = \{K\in\mesh\,:\,\edge\in\edgescv\}$.
We then assume that, for all $\edge\in\edges$, $\mesh_\edge$ has exactly one element
and $\edge\subset\partial\O$, or $\mesh_\edge$ has two elements and
$\edge\subset\O$. 
We let $\edgesint$ be the set of all interior faces, i.e. $\edge\in\edges$ such that $\edge\subset \O$, and $\edgesext$ the set of boundary
faces, i.e. $\edge\in\edges$ such that $\edge\subset \dr\O$.
For $\edge\in\edges$, the $(d-1)$-dimensional measure of $\edge$ is $\medge$,
the centre of gravity of $\edge$ is $\centeredge$

\item $\centers = (\xcv)_{K \in \mesh}$ is a family of points of $\O$ indexed by $\mesh$ and such that, for all  $K\in\mesh$,  $\xcv\in K$ ($\xcv$ is sometimes called the ``centre'' of $\cv$). 
We then assume that all cells $K\in\mesh$ are  strictly $\xcv$-star-shaped, meaning that 
if $\x$ is in the closure of $K$ then the line segment $[\xcv,\x)$ is included in $K$.

\item $\mathcal V$ is the set of vertices of the mesh. The vertices that belong to
$\overline{K}$, for $K\in\mesh$, are gathered in $\vertices_K$; the set of vertices of
$\edge\in\edges$ is denoted by $\vertices_\edge$.

\end{enumerate}
For all $K \in \mesh$ and for any $\edge \in \edgescv$, we denote by
 $\ncvedge$ the (constant) unit vector normal to $\edge$ outward to $K$.
We also let $\dcvedge$ be the signed orthogonal
distance between $\xcv$ and $\edge$ (see left part of Fig. \ref{fig.dksigma}), that is: 
\begin{equation}
 \dcvedge = (\x - \xcv) \cdot \ncvedge\,,\quad\forall \x \in \edge \label{def.dcvedge}
\end{equation}
(note that $(\x - \xcv) \cdot \ncvedge$ is constant for $\x \in \edge$).
The fact that $K$ is strictly star-shaped with respect to $\xcv$ is equivalent
to $\dcvedge> 0$ for all $\edge\in\edgescv$.
For all $K\in\mesh$ and $\edge\in\edgescv$, we denote by $D_{K,\edge}$ the cone with apex $\xcv$ and base $\edge$, that is $D_{K,\edge}=\{ t \xcv +(1-t) \y\,:\, t\in (0,1),\,
\y\in \edge\}$.
The diamond associated to a face $\edge\in\edges$ is $D_\edge = \bigcup_{K\in\mesh_\edge} D_{K,\edge}$.

The size of the discretisation is $\sizemesh=\sup\{h_K\,:\; K\in \mesh\}$ and the
regularity factor $\theta_\polyd $ is
\be\label{def:theta}
\theta_\polyd = \max\left\{\frac{h_K}{d_{K,\sigma}}+\frac{|K|}{|D_{K,\edge}|}\,:\,K\in\mesh\,,\;\sigma\in\edgescv\right\}
+\max\left\{\frac{d_{K,\edge}}{d_{L,\edge}}\,:\,\sigma\in\edgesint\,,\;\mesh_\sigma=\{K,L\}\right\}.
\ee
\end{definition}

\begin{figure}[htb]
\begin{center}
\resizebox{.3\textwidth}{!}{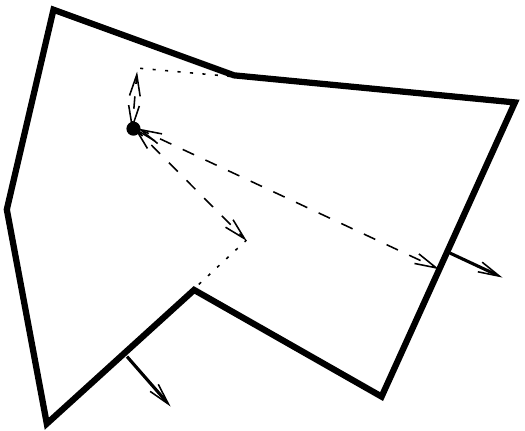}\hspace*{.2\textwidth}\resizebox{.3\textwidth}{!}{\includegraphics{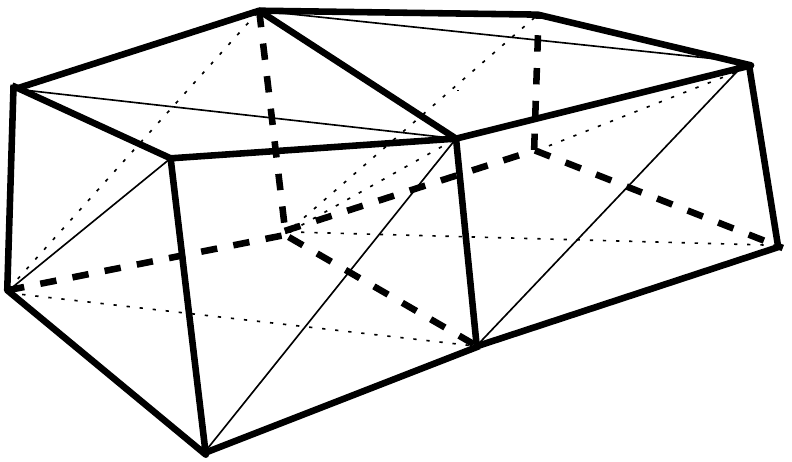}}
\caption{A cell $K$ of a polytopal mesh (left). Two neighbouring generalised hexahedra (right).}
\label{fig.dksigma}
\end{center}
\end{figure}

\begin{remark}[Generalised hexahedra] This definition covers a wide variety of meshes, including those with non-convex cells and cells sharing more than one face;
in particular, ``generalised hexahedra'' with non planar faces can be handled;
such cells have 12 faces (if each non planar face is split in two triangles), but only 6 neighbouring cells. See right of Fig. \ref{fig.dksigma}.
\end{remark}

\begin{remark} Since $\min_{\sigma \in\edgescv}d_{K,\edge}$ is smaller than the radius
of the largest ball centred at $\x_K$ and contained in $K$, an upper bound on
$\theta_\polyd$ imposes that the interior and exterior diameters of each
cell are comparable.
\end{remark}

We now introduce a ``polytopal toolbox'', used in the statement of discrete functional analysis results.

\begin{definition}[Polytopal toolbox] \label{def:polydis} 
Let $\Omega$ be a bounded polytopal open subset of $\R^d$ ($d\ge 1$) and let $\polyd$ be a polytopal mesh in the sense of Definition \ref{def:polymesh}. 
The quadruplet $(X_{\polyd,0},\Pi_\polyd,\nabla_\polyd,\Vert \cdot \Vert_{\polyd,0,p})$ is a polytopal toolbox if:
\begin{enumerate}
 \item the set $X_{\polyd,0}$ is the vector space of degrees of freedom attached to cells and edges (with homogeneous Dirichlet boundary conditions):
\begin{align}
% \label{dfa:spaceM}
% &X_\mesh=\{v=(v_K)_{K\in\mesh}\,:\, v_K\in\R\},\\
\label{dfa:spaceH}
&X_{\polyd,0}=\{v=((v_K)_{K\in\mesh},(v_\sigma)_{\sigma\in\edges})\,:\, v_K\in\R\,,\;
v_\sigma\in\R\,,\,v_\sigma=0\mbox{ if }\sigma\in\edgesext\}.
\end{align}
\item The mapping $\Pi_\polyd:X_{\polyd,0} \to L^\infty(\O)$ is defined by
\be\label{def:pipolyd}
\forall v\in X_{\polyd,0},\;\forall K\in \mesh,\;
\Pi_\polyd v=v_K\mbox{ on $K$}.
\ee
\item The discrete gradient $\nabla_\polyd: X_{\polyd,0}\mapsto L^p(\O)^d$ is defined by
\be\label{def:nablapolyd}
\forall K\in \mesh\,,\;\nabla_\polyd v=\frac{1}{|K|}\sum_{\edge\in\edgescv}|\sigma|(v_\sigma-v_K)\n_{K,\edge}=\frac{1}{|K|}\sum_{\sigma\in\edgescv}|\edge|v_\sigma\n_{K,\edge} \mbox{ on $K$}
\ee
(the second equality follows from the property $\sum_{\edge\in\edgescv}|\edge|\n_{K,\edge}=0$,
a consequence of Stokes' formula).
\item The space $X_{\polyd,0}$
is endowed with the following discrete $W^{1,p}_0$ norm,
for some $p\in (1,\infty)$:
\begin{align}
% \label{def:seminormM}
% \forall v\in X_\mesh\,:&\,||v||_{X_\mesh,p}^p=\sum_{\edge\in\edges}
% |\edge|d_\edge\left|\frac{v_K-v_L}{d_\edge}\right|^p,\\
\label{def:seminormD}
\forall v\in X_{\polyd,0}\,,\;
||v||_{\polyd,0,p}^p=\sum_{K\in\mesh}\sum_{\edge\in\edgescv}
|\edge|d_{K,\edge}\left|\frac{v_\edge-v_K}{d_{K,\edge}}\right|^p.
\end{align}
\end{enumerate}
In the sequel, $\polyd$ refers to both the polytopal mesh and to the quadruplet $(X_{\polyd,0},\Pi_\polyd,\nabla_\polyd,\Vert \cdot \Vert_{\polyd,0,p})$.
\end{definition}
The discrete gradient $\nabla_\polyd$ satisfies, thanks to H\"older's inequality and to $\sum_{\sigma\in\edgescv}|\edge|d_{K,\edge} = d|K|$,
\be\label{eq:majnordisc}
\Vert \nabla_\polyd v\Vert_{L^p(\O)^d} \le d^{\frac{p-1}{p}}||v||_{\polyd,0,p}.
\ee

\begin{remark} Note that a polytopal toolbox is not a gradient discretisation, since
$\Vert\nabla_\polyd \cdot\Vert_{L^p(\O)^d}$ is not a norm on $X_{\polyd,0}$ (consider $v\in X_{\polyd,0}$ such that $v_\edge=0$ for all $\edge\in\edges$ but $v_K\not=0$
for some $K\in\mesh$).
\end{remark}

The following lemmas, whose proof can be found in \cite{sushi,koala}, are used to
establish Proposition \ref{prop:inhdfa} below.

\begin{lemma}[Discrete Poincar\'e inequality]\label{dfa:discP}
Let $\polyd$ be a polytopal toolbox of $\O$ in the sense of Definition \ref{def:polydis}, and let $\theta\ge \theta_\polyd$.
There exists $\ctel{dfa:P}$ only depending on $\Omega$, $\theta$ and $p$ such that for
all $v\in X_{\polyd,0}$ we have $||\Pi_\polyd v||_{L^p(\O)}
\le \cter{dfa:P}||v||_{\polyd,0,p}$.
\end{lemma}

\begin{lemma}[Discrete Rellich theorem]\label{dfa:discR}
Let $(\polyd_m)_{m\in\N}$ be a sequence of polytopal toolboxes of $\O$
in the sense of Definition \ref{def:polydis}, such that $(\theta_{\polyd_m})_{m\in\N}$ is 
bounded. If $v_m\in X_{\polyd_m,0}$ is such that $(||v_m||_{\polyd_m,0,p})_{m\in\N}$
is bounded, then $(\Pi_{\polyd_m}v_m)_{m\in\N}$ is relatively compact in
$L^p(\O)$.
\end{lemma}

%\begin{remark} The discrete functional analysis tools in \cite{sushi,koala}
%are initially written for the space \eqref{dfa:spaceH}.
%As explained in \cite{koala}, these tools also apply to $X_\mesh$,
%that only has degrees of freedom at the cells. An easy reconstruction
%of face values allows us to embed $X_\mesh$ into the space defined
%by \eqref{dfa:spaceH}, and to use the results of \cite{koala} on this space.
%This is how lemmas \ref{dfa:discP} and \ref{dfa:discR} are deduced from
%results in \cite{koala}.
%\end{remark}

\begin{lemma}[Discrete approximate Stokes formula]
Let $\polyd$ be a polytopal toolbox of $\O$ in the sense of Definition \ref{def:polydis}.
If $\bvarphi\in C^\infty_c(\R^d)^d$ and $v\in X_{\polyd,0}$, then
\be\label{stokesapp}
\left|\int_\O \left[\nabla_\polyd v(\x)\cdot\bvarphi(\x) + \Pi_\polyd v(\x)
\div\bvarphi(\x)\right]\d\x\right|\le
(d|\O|)^{\frac{p-1}{p}}||\bvarphi||_{W^{1,\infty}(\R^d)^d}
||v||_{\polyd,0,p}h_\mesh.
\ee
Moreover, if $(v_\edge)_{\edge\in\edgescv}$ are the exact values
at $(\centeredge)_{\edge\in\edgescv}$ of an affine mapping $L$, then
$\nabla_\polyd v=\nabla L$ on $K$.
\label{lem:stokes}\end{lemma}

The preceding discrete functional analysis results are useful for the analysis
of a wide number of numerical methods, thanks to the notion of control
of gradient discretisations by polytopal toolboxes.

\begin{definition}[Control of a gradient discretisation by a polytopal toolbox]
\label{def:inhdfa}
Let $\Omega$ be a bounded polytopal open subset of $\R^d$ ($d\ge 1$), let $\disc$ be a gradient discretisation in the sense of Definition \ref{defgraddisc}, and let $\polyd$ be a polytopal toolbox of $\O$ in the sense of Definition \ref{def:polydis}. A control of $\disc$ by $\polyd$ is a linear mapping $\emb~:~X_{\disc,0}\longrightarrow X_{\polyd,0}$. We then define
\begin{align}
\nonumber
&||\emb||_{\disc,\polyd} = \max_{v\in X_{\disc,0}\setminus\{0\}}\frac {\Vert \emb(v)\Vert_{\polyd,0,p}} {\Vert \nabla_{\disc} v \Vert_{L^p(\O)^d}},
\\ \nonumber
& \omega^{\Pi}(\disc,\polyd,\emb) = \max_{v\in X_{\disc,0}\setminus\{0\}}\frac { 
\Vert \Pi_{\disc} v-\Pi_{\polyd}\emb(v)\Vert_{L^p(\O)} } {\Vert \nabla_{\disc} v \Vert_{L^p(\O)^d}},
\\ \nonumber
& \omega^\nabla(\disc,\polyd,\emb) = \max_{v\in X_{\disc,0}\setminus\{0\}}\frac { 
\dsp \sum_{K\in\mesh}\left| \int_K \left[\grad_{\disc} v(\x)-\grad_{\polyd}\emb(v)(\x)\right]\d\x\right|} {\Vert \nabla_{\disc} v \Vert_{L^p(\O)^d}}.
\end{align}
\end{definition}

In most of the examples of gradient discretisations in Section \ref{sec:review},
the following definition and proposition are used to establish the coercivity,
compactness and limit-conformity.

\begin{definition}[Regularity of a sequence of polytopal meshes] \label{def:regpolyd}
 A sequence of polytopal meshes $(\polyd_m)_{m\in\N}$ in the sense of Definition \ref{def:polymesh} is regular if $(\theta_{\polyd_m})_{m\in\N}$
is bounded and if $h_{\mesh_m}\to 0$ as $m\to+\infty$. 
\end{definition}
\begin{proposition}[Properties of gradient discretisations controlled by polytopal toolboxes]\label{prop:inhdfa}
Let $\Omega$ be a boun\-ded polytopal open subset of $\R^d$ ($d\ge 1$). 
Let $(\disc_m)_{m\in\N}$  be a sequence of gradient discretisations
in the sense of Definition \ref{defgraddisc}, and let $(\polyd_m)_{m\in\N}$ be a sequence of polytopal toolboxes of $\O$ in the sense of Definition \ref{def:polydis} such that the corresponding sequence of polytopal meshes is regular in the sense of Definition \ref{def:regpolyd}.
We take, for all $m\in\N$, a control $\emb_m$  of $\disc_m$ by $\polyd_m$ in the sense of Definition \ref{def:inhdfa}, and we assume that
\begin{align}
\label{cnt:poly1}
&\mbox{There exists } \CSTcontrol>0 \mbox{ such that, for all } m\in\N,\ ||\emb_m||_{\disc_m,\polyd_m}\le \CSTcontrol,
\\ \label{cnt:poly2}
& \lim_{m\to\infty} \omega^{\Pi}(\disc_m,\polyd_m,\emb_m) =0,
\\ \label{cnt:poly3}
& \lim_{m\to\infty} \omega^{\nabla}(\disc_m,\polyd_m,\emb_m) =0.
\end{align}
Then $(\disc_m)_{m\in\N}$ is coercive in the sense of Definition \ref{def-coer}, limit-conforming in the sense of Definition \ref{def-limconf}, and compact in the sense of  Definition \ref{def-comp}.
\end{proposition}

\begin{proof} We let $\omega_m = \max[\omega^{\Pi}(\disc_m,\polyd_m,\emb_m),\omega^{\nabla}(\disc_m,\polyd_m,\emb_m)]$ and $M = \max_{m\in\N} \omega_m$.

\textsc{Coercivity}: using \eqref{cnt:poly2} and Lemma \ref{dfa:discP},
we observe that, for any $v\in X_{\disc_m,0}$,
\[
||\Pi_{\disc_m} v||_{L^p(\O)^d} \le M ||\nabla_{\disc_m}v||_{L^p(\O)^d} + \Vert \Pi_{\polyd_m}\emb_m(v)\Vert_{L^p(\O)} \le M ||\nabla_{\disc_m}v||_{L^p(\O)^d} +
\cter{dfa:P}||\emb_m(v)||_{\polyd,0,p}.
\]
Property \eqref{cnt:poly1} therefore give
$||\Pi_{\disc_m} v||_{L^p(\O)^d} \le (M +  \cter{dfa:P} \CSTcontrol) ||\nabla_{\disc_m}v||_{L^p(\O)^d}$, and the coercivity follows.

\textsc{Limit-conformity}: as stated in Remark \ref{rem:eqdef}, since $\Omega$ is polytopal and therefore locally star-shaped, we only need to consider $\bvarphi\in C^\infty_c(\R^d)^d$. 
By the triangular inequality, \eqref{cnt:poly2} and \eqref{stokesapp}, we have
\begin{align}
\Bigg|\int_\O \Big(\grad_{\disc_m} &v(\x)\cdot\bvarphi(\x) + \Pi_{\disc_m}v(\x) \div\bvarphi(\x)\Big)  \d\x\Bigg|\nonumber\\
&\le \left|\int_\O [\grad_{\disc_m}v(\x) - \nabla_{\polyd_m} \emb_m(v)(\x)]\cdot\bvarphi(\x) \d\x\right|+
||\div\bvarphi||_{L^{p'}(\O)}\omega_m||\nabla_{\disc_m}v||_{L^p(\O)^d}\nonumber\\
&\qquad+\left|\int_\O \left[\nabla_{\polyd_m} \emb_m(v)(\x)\cdot\bvarphi(\x) + \Pi_{\polyd_m} \emb_m(v)(\x)
\div\bvarphi(\x)\right]\d\x\right|\nonumber\\
&\le \left|\int_\O [\grad_{\disc_m}v(\x) - \nabla_{\polyd_m} \emb_m(v)(\x)]\cdot\bvarphi(\x) \d\x\right|+
||\div\bvarphi||_{L^{p'}(\O)}\omega_m||\nabla_{\disc_m}v||_{L^p(\O)^d}\nonumber\\
&\qquad+ (d|\O|)^{\frac{p-1}{p}}C_\bvarphi
||\emb_m(v)||_{\polyd_m,0,p}h_{\mesh_m}
\label{emb:lc.1}\end{align}
where $C_\bvarphi=||\bvarphi||_{W^{1,\infty}(\R^d)^d}$.
We define $\bvarphi_K=\frac 1 {\mcv}\int_K \bvarphi(\x)  \d\x$ and notice that
$|\varphi_K|\le C_\bvarphi$ and $|\bvarphi(\x) - \bvarphi_K|\le C_\bvarphi h_{\mesh_m}$ for all $\x\in K$. Therefore, since $\nabla_{\polyd_m}\emb_m(v)$ is constant in each cell,
\begin{align*}
 \Bigg|\int_\O &[\grad_{\disc_m}v(\x) - \nabla_{\polyd_m} \emb_m(v)(\x)]\cdot\bvarphi(\x)  \d\x
\Bigg|= \Bigg|\sum_{K\in\mesh_m} \int_K[\grad_{\disc_m}v(\x) - \nabla_{\polyd_m} \emb_m(v)(\x)]\cdot\bvarphi(\x)  \d\x\Bigg|\\
&= \Bigg|\sum_{K\in\mesh_m} 
\left(\int_K \grad_{\disc_m}v(\x)\cdot
[\bvarphi(\x)-\bvarphi_K]  \d\x
+ \bvarphi_K\cdot \int_K (\grad_{\disc_m}v(\x) - \nabla_{\polyd_m} \emb_m(v)(\x))\d\x\right)\Bigg|\\
& \le C_\bvarphi \sum_{K\in\mesh_m} \left( h_{\mesh_m}\int_K|\grad_{\disc_m}v(\x)|\d\x +   \left|\int_K(\grad_{\disc_m}v(\x) -\nabla_{\polyd_m} \emb_m(v)(\x)) \d\x\right|\right)\\
  &\le C_\bvarphi \left( h_{\mesh_m}|\O|^{\frac{p-1}{p}} 
   +   \omega_m\right)\Vert \nabla_{\disc_m} v \Vert_{L^p(\O)^d}.
\end{align*}
We used H\"older's inequality and  \eqref{cnt:poly3} in the last line.
Plugged into \eqref{emb:lc.1} and using \eqref{cnt:poly1} this gives
\[
W_{\disc_m}(\bvarphi)
\le C_\bvarphi  \left( h_{\mesh_m}|\O|^{\frac{p-1}{p}}    +   \omega_m \right)+ ||\div\bvarphi||_{L^{p'}(\O)}\omega_m+
(d|\O|)^{\frac{p-1}{p}}C_\bvarphi \CSTcontrol h_{\mesh_m}.
\]
The limit conformity of $(\disc_m)_{m\in\N}$ follows.

\textsc{Compactness}: by \eqref{cnt:poly1}, if
$(\nabla_{\disc_m}v_m)_{m\in\N}$ is bounded in $L^p(\O)^d$ then
$ ||\emb_m(v_m)||_{\polyd_m,0,p}$ is bounded. Applying Lemma \ref{dfa:discR}, we obtain
the relative compactness of $(\Pi_{\polyd_m}\emb_m(v))_{m\in\N}$ in $L^p(\O)$. Since
$||\Pi_{\disc_m}v_m-\Pi_{\polyd_m}\emb_m(v)||_{L^p(\O)}\to 0$ as $m\to\infty$
by  \eqref{cnt:poly2}, we deduce that $(\Pi_{\disc_m}v_m)_{m\in\N}$ is relatively compact in $L^p(\O)$.
\end{proof}

\section{Review of gradient discretisations} \label{sec:review}

We now study a number of known methods among finite element, finite volume methods, mimetic  methods and related discretisation schemes which are all based on polytopal meshes.
Each of the following sections is devoted to a particular method which is shown to be the
gradient scheme of a gradient discretisation referred to as $\disc$;
for each method we define a regular sequence $(\disc_m)_{m\in \N}$ of gradient discretisations, based on the method itself and on the regularity of a polytopal mesh (Definition \ref{def:regpolyd}), and we show the following property.
\smallskip

\begin{center}~\hfill
\begin{minipage}{14cm}
 The regular sequence  $(\disc_m)_{m\in \N}$ is coercive, consistent, limit-conforming and compact in the sense of the definitions in Section \ref{sec:deftools}.
\end{minipage}
 \hfill$(\mathcal{P})$
\end{center}

\medskip

 The proof of $(\mathcal{P})$ relies on the notions of LLE gradient discretisations (section \ref{sec:LL}), barycentric condensation (section \ref{sec:bc}), mass lumping (section \ref{sec:masslump}) and polytopal toolbox (section \ref{sec:dfa}).
% 
% \begin{theorem}[Properties of the ``$\alpha$ method'']\label{theo:generic}
% Let $(\disc_m)_{m\in\N}$ be a regular sequence of gradient discretisations defined by the ``$\alpha$ method''.
% Then $(\disc_m)_{m\in\N}$ is coercive, consistent, limit-conforming and compact in the sense of the definitions in Section \ref{sec:deftools}.
% \end{theorem}
% 
% Hence this theorem only focuses on the four core properties, among the five ones, which need a proof. The piecewise constant reconstruction property is imposed by the construction of $\disc$ only in some of the following methods.
% 
%  

\subsection{$\P_k$  finite element methods}\label{sec:galerkin}

\subsubsection{Conforming methods: $\P_k$ finite elements}\label{sec:Pk}

Let $\polyd$ be a simplicial discretisation of $\O$, that is a polytopal discretisation in the sense of Definition \ref{def:polymesh} such that for any $K\in\mesh$ we have ${\rm Card}(\mathcal E_K)=d+1$. 
Let $k\in\N\setminus\{0\}$.
We follow Definition \ref{def:LL} for the construction of $\disc=(X_{\disc,0},\nabla_\disc,\Pi_\disc)$ by describing the partition of $\O$, the functions $\alpha_i$ and the
local linearly exact gradients in the elements of the partition.
\begin{enumerate}
 \item The set $I$ of geometrical entities attached to the dof is $I =\vertices^{(k)}$, 
 and the set of approximation points is $S=I$, where $\vertices^{(k)}=\bigcup_{K\in\mesh} \vertices^{(k)}_K$ and $\vertices^{(k)}_K$ is the set of the points $\x$ of the form
 \begin{equation}
  \x =\sum_{\vertex\in\verticescv} \frac {i_\vertex}{k} \vertex
\quad\mbox{ with }(i_\vertex)_{\vertex\in\verticescv}\in \{0,\ldots,k\}^{\verticescv}
\mbox{ such that }\sum_{\vertex\in\verticescv} i_\vertex = k.
 \label{def:vertexpk}\end{equation}
Then  $\Ii = \vertices^{(k)}_{\rm int}$ (the subset of the interior vertices) and $\Ib =  \vertices^{(k)}_{\rm ext}$ (boundary vertices), and the partition of $\O$ is given by $\partition = \mesh$. For $U=K\in\partition$, we let $I_U = \vertices^{(k)}_K$.

 \item The reconstruction operator $\Pi_\disc$ in \eqref{ll:piD} is defined using the basis functions
$(\alpha_\vertex)_{\vertex\in\vertices_K^{(k)}}$, called in this particular case the Lagrange interpolation operators and defined the
following way: in each cell $K$, $\alpha_\vertex$ 
is the polynomial function of $\x$ with degree $k$, such that $\alpha_\vertex(\vertex) = 1$ and $\alpha_\vertex(\vertex') = 0$ for all $\vertex'\in \vertices^{(k)}_K\setminus\{\vertex\}$. This leads to
 \[
  \forall v\in X_{\disc,0},\ \forall \x\in \O,\  \Pi_\disc v(\x) = \sum_{\vertex\in\vertices^{(k)}} v_\vertex \alpha_\vertex(\x).
 \]

  \item The linearly exact gradient reconstruction in $K$ is 
 \[
 \forall \x\in K,\  \gr_K v(\x) = \sum_{\vertex\in\verticescv^{(k)}} v_\vertex \grad \alpha_\vertex(\x)=\nabla(\Pi_\disc v)(\x).
 \]
 \item
 We have $\Pi_\disc v \in W^{1,p}_0(\O)$ so
the Poincar\'e inequality in $W^{1,p}_0(\Omega)$ implies that $\Vert \nabla_\disc \cdot \Vert_{L^p(\Omega)^d}$ is a norm on $X_{\disc,0}$.
\end{enumerate}

Defining the regularity of a sequence of $\P_k$ discretisations $(\disc_m)_{m\in\N}$ merely as the regularity of the underlying polytopal meshes (Definition \ref{def:regpolyd}) is sufficient to 
obtain the boundedness of $(\regLLE(\disc_m))_{m\in\N}$; hence  Proposition \ref{prop:LL} implies the consistency.
Since $\nabla \Pi_\disc v=\nabla_\disc v$ we have $W_\disc\equiv 0$ and the limit conformity is
trivial; the coercivity and the compactness are consequences of the Poincar\'e inequality
and the Rellich theorem in $W^{1,p}_0(\O)$ respectively.
This establishes $(\mathcal P)$ for $\P_k$ gradient discretisations.

\subsubsection{Mass-lumped $\P_1$ finite elements}  

\begin{figure}
\begin{center}
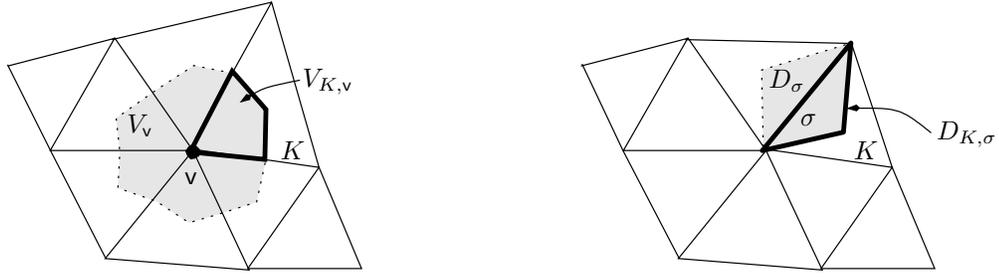
\caption{Partitions for mass-lumping of the $\P_1$ (left) and
non-conforming $\P_1$ (right) finite element methods.}\label{fig:ml}
\end{center}
\end{figure}

We construct a mass-lumped version of the $\P_1$ gradient discretisation
as per Definition \ref{def:ml}, with the natural geometrical entities attached to the dof $\vertices^{(1)} = \vertices$. 
Subdomains $(V_\vertex)_{\vertex\in \vertices}$
with points of $\vertices$ as centres can be constructed in various ways. One way
is to define $V_\vertex$ as the set of all $\y\in \O$
such that $\alpha_\vertex(\y)> \alpha_{\vertex'}(\y)$ for any other $\vertex'\in \vertices$.
The left part of Figure \ref{fig:ml} illustrates the construction
of the partitions $(V_\vertex)_{\vertex\in \vertices}$ in the case $d=2$ (then $(V_\vertex)_{\vertex\in\vertices}$ is sometimes called  the barycentric dual mesh of  $\polyd$).

A Taylor expansion in each $V_\vertex\cap K$ then shows that Estimate \eqref{est:ml} holds with $\omega_m = h_{\mesh_m}$,
and thus by Theorem \ref{th:ml} we see that the mass-lumped $\P_1$ gradient discretisation
satisfies the property $(\mathcal{P})$.

\subsection{Non-conforming $\P_1$ finite elements} \label{sec:ncPk}

\subsubsection{Standard non-conforming $\P_1$ reconstruction}

Non-conforming $\P_1$ finite elements consist in approximating the
solution to \eqref{pblinw} by functions that are piecewise linear on
triangles and continuous at the edge midpoints -- but not necessarily continuous
on the whole edge. These approximating functions therefore do not lie in
$H^1_0(\O)$, and do not satisfy the exact Stokes formula; hence the name
``non-conforming''. 

Let $\polyd$ be a simplicial mesh of $\O$, that is a polytopal mesh in the sense of Definition \ref{def:polymesh} such that for any $K\in\mesh$ we have ${\rm Card}(\mathcal E_K)=d+1$. 
We refer to Definition \ref{def:LL} for the construction of $\disc$.

\begin{enumerate}
 \item The set of  geometrical entities attached to the dof is $I = \edges$ and
the approximation points are $S=(\centeredge)_{\edge\in\edges}$. Then  $\Ii = \edgesint$ and $\Ib =  \edgesext$, and the partition of $\O$ is given by $\partition = \mesh$.
For all $U=K\in\partition$, we let $I_U = \edgescv$.
 
 \item The reconstruction $\Pi_\disc$ in \eqref{ll:piD} is defined using the affine non-conforming finite element basis functions $(\alpha_\edge)_{\edge\in\edges}$
defined by: $\alpha_\edge$ is linear in each simplex, $\alpha_\edge(\centeredge)=1$ and
$ \alpha_\edge(\overline{\x}_{\edge'})=0$ for all $\edge'\in \edges\backslash\{\edge\}$.
This leads to
 \[
  \forall v\in X_{\disc,0},\ \forall \x\in \O,\  \Pi_\disc v(\x) = \sum_{\edge\in\edges} v_\edge \alpha_\edge(\x).
 \]

 \item The linearly exact gradient reconstruction in $K$ is defined by the constant value
 \[
 \forall \x\in K\,,\;  \gr_K v(\x) = \sum_{\edge\in\edgescv} v_\edge \grad \alpha_\edge(\x).
 \]
 
 %It satisfies $\nabla\Pi_\disc v=\nabla_\disc v$ in each $K\in\mesh$.
\item The fact that $\Vert \nabla_\disc \cdot \Vert_{L^p(\Omega)^d}$ is a norm on $X_{\disc,0}$ is deduced from the injectivity of the mapping $\emb$, defined in the course of the proof of the property $(\mathcal{P})$ below.

\end{enumerate}

The regularity of the non-conforming  $\P_1$  gradient discretisations is then defined as the regularity of the underlying polytopal discretisations $(\polyd_m)_{m\in\N}$ (see Definition \ref{def:regpolyd}).

\begin{proof}[Proof of the property $(\mathcal{P})$ for non-conforming $\P_1$ gradient
discretisations]

We drop the index $m$ from time to time for sake of legibility, and
all constants below do not depend on $m$ or the considered
cells/edges. 
Let us define a control of $\disc$ by $\polyd$ in the sense of Definition \ref{def:inhdfa}, where $\polyd$ is the simplicial mesh associated to $\disc$,
with $\xcv = \overline{\x}_K =\frac{1}{d+1}\sum_{\edge\in\edgescv}\centeredge$
the centres of gravity of the cells $K$. 
  We define the linear (injective) mappings $\emb~:~X_{\disc_m,0}
\longrightarrow X_{\polyd_m,0}$
by $\emb(u)_K = \frac{1}{d+1}\sum_{\edge\in\edgescv} u_\edge=\Pi_\disc u(\x_K)$ and $\emb(u)_\edge = u_\edge
=\Pi_\disc u(\centeredge)$.

Since $\emb(u)_K = \Pi_\disc u(\xcv)$ and $\gr_K u=\nabla(\Pi_\disc u)$ in $K$, we get
\begin{equation}
 \emb(u)_\edge - \emb(u)_K = \gr_K u\cdot (\centeredge - \xcv).
\label{eq:punnc}\end{equation}
Therefore, since $\frac{|\centeredge - \xcv|}{d_{K,\edge}} \le \frac{h_K}{d_{K,\edge}}\le \theta_{\polyd}$,
\[
\sum_{\edge\in\edgescv}\medge d_{K,\edge} \left|\frac {\emb(u)_\edge - \emb(u)_K} {\dcvedge}\right|^p \le \theta_{\polyd}^p d \mcv\, |\gr_K u|^p.
\]
This implies \eqref{cnt:poly1}. We now observe that the affine function $\alpha_\edge$ reaches its extremal values at the vertices of $K$. It is easy to see that $\alpha_\edge(\vertex_\edge) = 1 - d$, where $\vertex_\edge$ is the vertex opposite to the face $\edge$, and that $\alpha_\edge(\vertex_{\edge'}) = 1$ for all $\edge'\neq\edge$.
Therefore, for $\x\in K$,
\[
|\Pi_\disc u(\x) -\emb(u)_K| = \left|\sum_{\edge\in\edgescv} ( \emb(u)_\edge - \emb(u)_K) \alpha_\edge(\x)\right| \le (d+1)\max(1,d-1) \max_{\edge\in\edgescv} |\gr_K u\cdot (\centeredge - \xcv)|.
\]
This inequality implies $\omega^\Pi(\disc,\polyd,\emb) \le (d+1)\max(1,d-1)  h_{\mesh}$ and therefore \eqref{cnt:poly2} holds. Finally, recalling that $\Pi_\disc u$ is affine
in each simplex $K$ and that $\nabla_\polyd$ is exact on interpolants of affine
functions (cf. Lemma \ref{lem:stokes}), we see that $\grad_\disc u = \grad_\polyd \emb(u)$ in $\O$.
Hence $\omega^\nabla(\disc,\polyd,\emb) =0$ and \eqref{cnt:poly3} holds.
Proposition \ref{prop:inhdfa} therefore shows that $(\disc_m)_{m\in\N}$ is coercive in the sense of Definition \ref{def-coer}, limit-conforming in the sense of Definition \ref{def-limconf}, and compact in the sense of  Definition \ref{def-comp}.

Since non-conforming $\P_1$ gradient discretisations are LLE gradient discretisations, the
consistency of $(\disc_m)_{m\in\N}$ follows from Proposition \ref{prop:LL}
by noticing that $\regLLE(\disc_m)$ is controlled by $\theta_{\polyd_m}$.
\end{proof}

\subsubsection{Mass-lumped non-conforming $\P_1$ reconstruction}

Let us recall that, in the case $d=2$, for all pair $(\edge,\edge')\in\edges^2$ with $\edge\neq\edge'$, there holds
\[
 \int_\O \alpha_\edge(\x)\alpha_{\edge'}(\x)\d\x = 0.
\]
This property ensures that the non-conforming $\P_1$ method has a diagonal mass matrix, which is useful for computing \eqref{time}. However, Property \eqref{def:proppc} is not satisfied, 
which might prevent the usage of the non-conforming $\P_1$ scheme for some
nonlinear problems. To recover a piecewise constant reconstruction, and thus
\eqref{def:proppc}, we apply to the preceding gradient discretisation the mass lumping
process as in Definition \ref{def:ml}.
Recalling that the set of geometrical entities attached to the dof is $I=\edges$,
we define the subdomains $(V_i)_{i\in I}$ of Definition \ref{def:ml} as the diamonds $(D_\edge)_{\edge\in\edges}$ around the edges.
The right part of Figure \ref{fig:ml} illustrates the construction of this partition.
Since $\Pi_\disc v$ is linear and $\nabla(\Pi_\disc v)=\nabla_\disc v$
in each cell, and since $\Pi_{\disc^\ml}v=\Pi_\disc v(\centeredge)$
on $D_\edge$, an order one Taylor expansion immediately provides Estimate \eqref{est:ml}.
Property $(\mathcal{P})$ for the mass-lumped non-conforming $\P_1$ gradient discretisation is
then a consequence of Theorem \ref{th:ml}.

\subsection{Mixed finite element $\RTk$ schemes} \label{sec:rtkmfe}

The $\RTk$ method is the only one presented here for which the gradient discretisation
is only constructed for $p=2$. All the other gradient discretisations are constructed
for any $p\in (1,\infty)$.

Let $\polyd$ be a simplicial discretisation of $\O$ as for the non-conforming
$\P_1$ scheme.
We fix $k\in \N$ and introduce the following spaces 
\begin{align*}
\bfV_h ={}& \{\bfw \in (L^2(\Omega))^d  \,:\, \bfw|_K \in \RTk(K), \ \forall K \in \mesh\},
\quad &\bfV_h^{\rm div} ={}&  \bfV_h  \cap H_{\rm div}(\Omega),\\
W_h ={}& \{p \in L^2(\Omega)\,:\,  p|_K \in \P_k(K), \ \forall K \in \mesh\},\quad
&M_h^0 ={}&  \left\{\mu\,:\,\bigcup_{\edge \in \edges} \overline \edge \to \R, \mu|_\edge \in \P_k(\edge), \mu|_{\partial \Omega} =0 \right\},
\end{align*}
where 
\begin{itemize}
\item $ \P_k(K)$ is the space of polynomials of $d$ variables on $K$ of degree less than or equal to $k$.%; its  dimension is $\left(\substack{ k+d \\ d }\right)$.
\item $ \P_k(\edge)$ is the space of polynomials of $d-1$ variables on $\edge$ of degree less than or equal to $k$.%; its  dimension is  $ \left(\substack{ k+d-1 \\ d-1 }\right)$.
\item $\RTk(K)=\P_k(K)^d + \x \overline{\P}_k(K)$
is the Raviart-Thomas space of order $k$ defined on $K$.
Here, $\overline{\P}_k(K)\subset \P_k(K)$ is the
set of homogeneous polynomials of degree $k$.
%; its dimension is $d \left(\substack{ k+d \\ d }\right) +   \left(\substack{ k+d-1 \\ d-1 }\right)$ 
\end{itemize}
We construct a gradient discretisation (for $p=2$ only) inspired by the dual mixed finite element formulation of Problem \eqref{pblin} as in \cite{arnold1985mixed}. 
Assuming that $A$ is constant in each cell $K$, the dual mixed finite element formulation
of \eqref{pblin} is
\be
\begin{array}{l}
\dsp (\bv ,q,\lambda) \in   \bfV_h  \times W_h \times M_h^0, \\  
\dsp \int_K \bw(\x)\cdot A(\x)^{-1} \bv(\x) \d\x - \int_K q(\x)\div\bw(\x)\d\x  + \sum_{\edge \in \edgescv}\int_\edge \lambda(\x) \bw|_K(\x) \cdot \ncvedge \d s(\x) =0,\;\forall \bw\in \bfV_h , \\
\dsp \int_K \psi(\x)\div \bv(\x)\d\x=\int_K \psi(\x)f(\x)\d\x,\; \forall \psi \in  W_h, \forall K \in \mesh, \\
\dsp \int_\edge \mu (\x)  \bv|_K(\x) \cdot \ncvedge \d s(\x) + \int_\edge \mu(\x)  \bv|_L(\x) \cdot \bfn_{L,\sigma} \d s(\x) =0, \; \forall \edge\in \edgesint\mbox{ with $\mesh_\edge=\{K,L\}$},  \, \forall \mu \in M_h^0.
\end{array}
\label{mixtgs:eq:ab}
\ee
We again refer to Definition \ref{def:LL} for the construction of $\disc=(X_{\disc,0},\nabla_\disc,\Pi_\disc)$. 
We consider $(\psi_i)_{i\in I^W}$ the standard basis of $W_h$, and
$(\xi_j)_{j\in I^M}$ the standard basis of $M_h^0$. 
These two standard bases are respectively associated to the set of points $I^W$ located in the cells and the set of points $I^M$ located on the faces of the cells.
These points are defined in a similar way as \eqref{def:vertexpk}.
\begin{enumerate}
 \item The set of geometrical entities attached to the dof is $I = I^W\cup I^M$, and the set of approximation points is also $S=I^W\cup I^M$.
 Then  $\Ii = I^W\cup I^M_{\rm int}$ and $\Ib =  I^M_{\rm ext}$, where $I^M_{\rm int}=I^M\cap \Omega$ and $I^M_{\rm ext}=I^M\cap\partial\Omega$.
 The partition of $\O$ is given by $\partition = \mesh$. 
 We denote by $I^W_K$ the set of all  points of $I^W$ which are in $K$, and by $I^M_{\edge}$ the set of all points of $I^M$ which are in $\edge$. 
 Then, for all $U=K\in\partition$, $I_U = I^W_K \cup \bigcup_{\edge\in\edgescv}I^M_{\edge} $.
 \item The reconstruction \eqref{ll:piD} is applied with $\alpha_i= \psi_i$ for all $i\in I^W_K$ and $\alpha_i = 0$ for all $i\in \bigcup_{\edge\in\edgescv} I^M_{\edge}$.
  This leads to
 \[
  \forall v\in X_{\disc,0},\ \forall \x\in \O,\  \Pi_\disc v(\x) = \sum_{i\in I^W} v_i \psi_i(\x).
 \]
 \item For all $K\in\mesh$, the linearly exact gradient reconstruction is locally defined in $K$ by: $\gr_K v$ is the function such that $A \gr_K v \in \RTk(K)$ and
 \begin{multline*}
 \forall\bw\in \RTk(K)\,,\; \int_K \bw(\x)\cdot \gr_K v(\x) \d\x + \int_K  \left(\sum_{i\in I^W_K} v_i \psi_i(\x)\right)\div\bw(\x)\d\x \\ - \sum_{\edge \in \edgescv}\int_\edge \left(\sum_{j\in I^M_\edge} v_j \xi_j(\x)\right) \bw|_K(\x) \cdot \ncvedge \d\gamma(\x) =0.  
 \end{multline*}
 \item  
 If $\Vert \nabla_\disc u \Vert_{L^p(\Omega)^d}=0$ then $\gr_K u= 0$, and $(\bv,q,\lambda)$ defined by $ \bv|_K= A \gr_K u$, $q= \sum_{i\in I^W} u_i \psi_i$, $\lambda = \sum_{j\in I^M} u_j \xi_j$ is a solution to \eqref{mixtgs:eq:ab} with $f = 0$.
 The invertibility of this system implies that $q=0$ and $\lambda=0$, and therefore $u_i= 0$ for $i \in I^W$ and $u_j = 0$ for $j \in I^M$.
  Therefore $\Vert \nabla_\disc \cdot \Vert_{L^p(\Omega)^d}$ is a norm on $X_{\disc,0}$.

\end{enumerate}
The proof of the equivalence between the corresponding gradient scheme \eqref{gslin}
and the Arnold-Brezzi mixed hybrid formulation \eqref{mixtgs:eq:ab} is found in \cite{egh2015rtk}, along with the proof of the property $(\mathcal{P})$ for a regular sequence of polytopal meshes in the sense of Definition \ref{def:regpolyd}.  Note that, in the case $k=0$, the property of piecewise constant reconstruction holds.

\subsection{Multi-point flux approximation MPFA-O scheme}\label{sec:mpfa}

We consider in this section two particular cases of the MPFA-O scheme
\cite{aav-96-dis}. They are based on particular polytopal meshes of $\O$ in the sense of Definition \ref{def:polymesh}: Cartesian for the first case, and simplicial for the second case. In each of these cases, for $K\in\mesh$ we let $\x_K = \overline{\x}_K$ be the centre of
gravity of $K$ and we define a partition $(V_{K,\vertex})_{\vertex\in\vertices_K}$ of $K$
the following way (see Figure \ref{fig:mpfa}):
\begin{itemize}
\item \emph{Cartesian meshes}: $V_{K,\vertex}$ is the parallelepipedic polyhedron whose faces
are parallel to the faces of $K$ and that has $\x_K$ and $\vertex$ as vertices. We define, for $\edge\in\edges$ and $\vertex\in\vertices_\edge$,  $\x_{(\edge,\vertex)} = \centeredge$ (note that these points are identical for all $\vertex\in\vertices_\edge$, see Remark \ref{rem:identicalpoints}).
\item \emph{Simplicial meshes}: we denote by $(\beta_{\vertex}^K(\x))_{\vertex\in\vertices_K}$
the barycentric coordinates of $\x$ in $K$ (that is
$\x-\xcv = \sum_{\vertex\in\vertices_K}\beta_{\vertex}^K(\x) (\vertex' - \xcv)$,
$\beta_{\vertex}^K(\x)\ge 0$ and $\sum_{\vertex'\in\vertices_K}\beta_{\vertex'}^K(\x) = 1$)
and we define $V_{K,\vertex}$ as the set of $\x\in K$ whose
barycentric coordinates $(\beta_{\vertex'}^K(\x))_{\vertex'\in\vertices_K}$ satisfy $\beta_\vertex^K(\x) > \beta_{\vertex'}^K(\x)$ for all $\vertex'\in \vertices_K\setminus\{\vertex\}$.
For $\edge\in\edges$ and $\vertex\in\vertices_\edge$,
$\x_{(\edge,\vertex)}$ is the point of $\edge$ whose  barycentric coordinates in  $\edge$ are $\beta_{\vertex'}^{\edge}(\x_{(\edge,\vertex)}) = 1/(d+1)$ for all $\vertex'\in\vertices_\edge\setminus\{\vertex\}$,
and $\beta_\vertex^{\edge}(\x_{(\edge,\vertex)}) = 2/(d+1)$.
\end{itemize}

\begin{figure}[!ht]
\centering
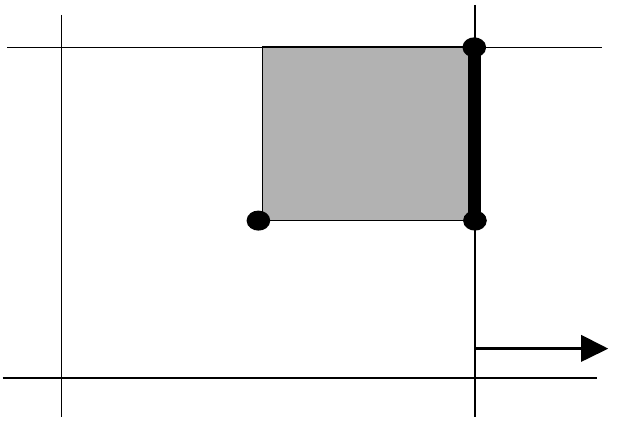\hspace*{0.1\linewidth}
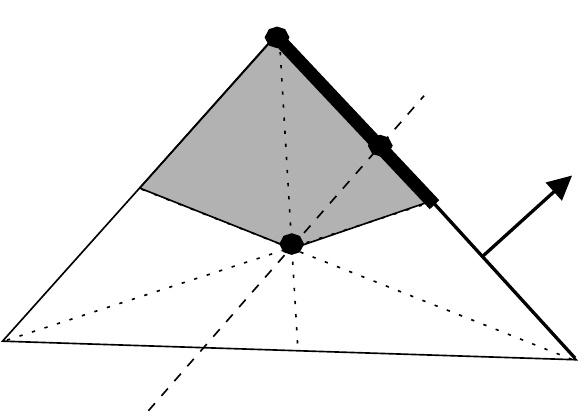
\caption{Notations for MPFA-O schemes defined on Cartesian (left) and simplicial (right) meshes.}
\label{fig:mpfa}
\end{figure}

We then follow the notations in Definition \ref{def:LL} to
construct the MPFA-O gradient discretisations in both cases:
\begin{enumerate}
 \item The set of geometrical entities attached to the dof is $I = \mesh \cup \{(\edge,\vertex)\,:\,\edge\in\edges,\ \vertex\in \vertices_\edge\}$ and the family of approximation points is
$S = ((\xcv)_{K\in\mesh},(\x_{(\edge,\vertex)})_{\edge\in\edges,\,\vertex\in \vertices_\edge})$.
We define $\Ii=\mesh\cup \{(\edge,\vertex)\,:\,\edge\in\edgesint,\ \vertex\in \vertices_\edge\}$ and $\Ib=\{(\edge,\vertex)\,:\,\edge\in\edgesext,\ \vertex\in \vertices_\edge\}$. 
The partition is $\partition = (V_{K,\vertex})_{K\in\mesh,\,\vertex\in\verticescv}$.
 For any $U = V_{K,\vertex}$, we set $\edges_{\cv,\vertex} = \{\edge\in\edgescv\,:\,\vertex\in\vertices_\edge\}$ and $I_U = \{K\}\cup \{(\edge,\vertex)\,:\,\edge\in\edges_{\cv,\vertex}\}$.
 \item The functions $\alpha_i$ are defined  by $\alpha_i = 1$ for $i=K$ and $\alpha_i = 0$ for $i=(\edge,\vertex)$, which means that
 \[
  \forall v\in X_{\disc,0}\,,\ \forall K\in\mesh\,,\ \forall \x\in K\,,\  \Pi_\disc v(\x) = v_K.
 \]
 \item Setting $\edge_{\vertex}=\overline{V_{K,\vertex}}\cap\edge$, the gradient reconstruction on $U = V_{K,\vertex}$ is
 \[
 \forall \x\in V_{K,\vertex}\,,\  \gr_{V_{K,\vertex}} v(\x) = \frac 1 {|V_{K,\vertex}|} \sum_{\edge\in\edges_{\cv,\vertex}} |\edge_{\vertex}| 
 (v_{(\edge,\vertex)} - v_K) \ncvedge.
 \]
 \item As in the case of the non-conforming $\P_1$ element, the fact that $\Vert \nabla_\disc \cdot \Vert_{L^p(\Omega)^d}$ is a norm on $X_{\disc,0}$ is deduced from the injectivity of the mapping $\emb$ defined in the course of the proof of the property $(\mathcal{P})$ below.
   
\end{enumerate}
 
For such a gradient discretisation, the gradient scheme \eqref{gslin} is a finite volume scheme. Indeed,
by selecting a test function with only non-zero value $v_K = 1$ in \eqref{gslin}, we obtain the flux balance
 \begin{equation}
 \sum_{\edge\in\edgescv}\sum_{\vertex\in\vertices_\edge} F_{K,\edge,\vertex}(u) = \int_K f(\x) \d \x, \mbox{ where }F_{K,\edge,\vertex}(u) = \int_{\edge_{\vertex}} \gr_{V_{K,\vertex}} u(\x)\cdot\ncvedge \d \gamma(\x). \label{eq:mpfacv}
 \end{equation}
Selecting a test function with only non-zero value $v_{(\edge,\vertex)} = 1$ in \eqref{gslin}
leads to the conservativity of the fluxes:
 \begin{equation}
  F_{K,\edge,\vertex}(u) + F_{L,\edge,\vertex}(u)= 0\hbox{ for all $\edge\in \edgesint$ with
$\mesh_\edge=\{K,L\}$, and all $\vertex\in\vertices_\edge$}.\label{eq:mpfacons}
 \end{equation}
We can also locally express the degree of freedom $u_{(\edge,\vertex)}$ in terms
of $(u_K)_{K|\vertex\in\vertices_K}$. For a given $\vertex\in\vertices$ 
this is done by solving the local linear system issued from \eqref{eq:mpfacons} written for all
$\edge$ such that $\vertex\in\vertices_\edge$. After these local eliminations of $u_{(\edge,\vertex)}$,
the resulting linear system only involves the cell unknowns.
This discretisation of \eqref{pblin} by writing the balance and
conservativity of half-fluxes $F_{K,\edge,\vertex}$, constructed
via a local linearly exact gradients, is identical to the construction
of the MPFA-O method in \cite{aav-96-dis}. This demonstrates that the
gradient discretisation constructed above indeed gives the MPFA-O method
when used in the gradient scheme \eqref{gslin}.

\begin{remark}
The identification of MPFA-O schemes as gradient schemes is, to our knowledge, restricted to the two cases considered in this section (Cartesian and simplicial meshes). 
In the case of more general meshes for the approximation of \eqref{pblin}, the discrete gradient defined by the MPFA-O scheme can be used in the finite volume scheme \eqref{eq:mpfacv}-\eqref{eq:mpfacons}; however, the gradient scheme \eqref{gslin} built upon this discrete gradient cannot be expected to converge, since the corresponding gradient discretisation may fail to be limit-conforming and coercive.
\end{remark}

The regularity of the MPFA-O gradient discretisations is defined as the regularity of $\polyd$ (see Definition \ref{def:regpolyd}).
Although references \cite{eymard2012appr,eymard2015appl} include proofs of the property $(\mathcal{P})$, let us show how the generic tools presented in Section \ref{sec:dfa}
enable very quick proofs of this result.

\begin{proof}[Proof of the property $(\mathcal{P})$ for MPFA-O gradient
discretisations]

We drop the indices $m$ for sake of legibility. 
We consider the polytopal mesh $\polyd = (\mesh,\edges',\centers,\vertices')$  where the sets $(\mesh,\centers)$ are those of the original polytopal mesh,  $\edges'=\{\edge_{\vertex}\,;\, \edge\in\edges,\ \vertex\in\vertices\} $, and $\vertices'$ is the set of all vertices of the elements of $\edges'$.
We define a control of $\disc$ by $\polyd$ (in the sense of Definition \ref{def:inhdfa})
as the isomorphism $\emb:X_{\disc,0} \longrightarrow X_{\polyd,0}$ given by
$\emb(u)_K = u_K$ and $\emb(u)_{\edge_\vertex} = u_{(\edge,\vertex)}$.
We observe that
\[
 \int_K |\grad_\disc u(\x)|^p\d\x \ge \cter{cte:mpfa} \sum_{\edge\in\edgescv}\sum_{\vertex\in\vertices_\edge} |\edge_\vertex|\dcvedge \left|\frac {u_{(\edge,\vertex)} - u_K} {\dcvedge}\right|^p,
\]
with $\ctel{cte:mpfa} = 1$ for parallelepipedic meshes, and  $\cter{cte:mpfa}>0$ depends on
an upper bound of the regularity  of the mesh for simplicial meshes.
Therefore $\Vert\grad_\disc u\Vert_{L^p(\O)^d}^p  \ge \cter{cte:mpfa}\Vert \emb(u)\Vert_{\polyd,0,p}^p$ and \eqref{cnt:poly1} is proved. Since 
$\Pi_{\disc} u = \Pi_{\polyd}\emb(u)$, we get $\omega^\Pi(\disc,\polyd,\emb) =0$,
which proves \eqref{cnt:poly2}. Finally, we have 
\[
 \int_K \grad_\disc u(\x)\d\x  = \sum_{\edge\in\edgescv}\sum_{\vertex\in\vertices_\edge} |\edge_\vertex|(u_{\edge,\vertex} - u_K)\ncvedge =
\sum_{\edge'\in\edges'_K}|\edge'|(\emb(u)_{\edge'}-\emb(u)_K)\n_{K,\edge'}=
 \mcv\;\grad_\polyd \emb(u)_{|K}.
\]
This shows that $\omega^\nabla(\disc,\polyd,\emb)=0$, which
establishes \eqref{cnt:poly3}.
Proposition \ref{prop:inhdfa} therefore shows that $(\disc_m)_{m\in\N}$ is coercive in the sense of Definition \ref{def-coer}, limit-conforming in the sense of Definition \ref{def-limconf}, and compact in the sense of  Definition \ref{def-comp}.

It is proved in \cite{eymard2012appr,eymard2015appl} that the definitions of the  approximation points $S$ give the LLE property in both the Cartesian and simplicial cases. 
Hence, the consistency of $(\disc_m)_{m\in\N}$ follows from Proposition \ref{prop:LL}. \end{proof}

\subsection{Discrete duality finite volumes}\label{sec:ddfv}

The principle of discrete duality finite volume (DDFV) schemes 
\cite{her-03-app,dom-05-fin,boy-08-fin,her-07-app,and-08-fin,cou-10-dis} is to design discrete divergence and gradient operators that are linked in duality through a discrete Stokes formula. 
Since discrete operators and an asymptotic Stokes
formula are at the core of gradient schemes (see Definitions \ref{defgraddisc}
and \ref{def-limconf}), it is not a surprise that they should
contain DDFV methods. This was already noticed without proof in \cite{eym-12-sma}; we give here a precise construction and proof of this result in the 3D case.
Note that the same tools can be used for the 2D case \cite{koala}.

Two 3D DDFV versions have been developed: the CeVe-DDFV, which uses cell and vertex unknowns
\cite{her-09-fin,cou-09-3d,and-10-dis}, and the CeVeFE-DDFV, which uses
cell, vertex, faces and edges unknowns \cite{cou-10-dis,cou-11-3d}.
The coercivity properties of the two methods differ:
the CeVe-DDFV does not seem to be unconditionally coercive on generic
meshes, whereas the CeVeFE-DDFV scheme is unconditionally coercive \cite{review}.
We show here that this latter method is a gradient scheme. To do so, we
introduce a gradient discretisation on a general octahedral mesh (possibly
including degenerate cells), and we show that when the octahedral cells of this mesh
are the ``diamond cells'' of a CeVeFE-DDFV method, the
gradient scheme corresponding to this gradient discretisation is the
CeVeFE-DDFV scheme. 
The standard CeVeFE-DDFV scheme corresponds to hexahedral meshes, which can be seen
as degenerate octahedral meshes (each cell has six vertices, but three of them are aligned
so only six physical faces are apparent). Although it was known to specialists that,
as done here, the construction could be performed on general octahedral meshes
(and that the corresponding DDFV method satisfies the discrete duality formula \cite{ABprivate}),
this was not reported before. It should also be noticed that our presentation gives
a complete description of the CeVeFE-DDFV method using only one mesh instead of the
usual four meshes. As shown in \cite[Section IX.B]{and-07-dis} for 2D DDFV methods,
the other three meshes can be reconstructed from the ocatahedral (``diamond'') mesh.
However these three meshes are not used to construct the method here.
The vision of DDFV methods based solely on one mesh (the ``diamond'' mesh, or octahedral mesh here) actually corresponds to the vision adopted in the implementation of the schemes.

Let %\label{ddfv:polyd} 
$\polyd=(\mesh,\edges,\centers,\vertices)$ be a polytopal mesh of $\O$ in the sense of Definition \ref{def:polymesh}, such that the elements of $\mesh$ are  octahedra (open polyhedra with eight triangular faces and six vertices, not necessarily convex; five vertices may be coplanar), and the element of $\edges$ are the triangular faces of the elements of $\mesh$. Each $\edgescv$ has 8 elements, each $\verticescv$  has $6$ elements, and each $\vertices_\edge$ has $3$ elements. 
For any $K\in \mesh$, the centre of $K$ is defined by $\xcv = \frac 1 6 \sum_{\vertex \in \verticescv} \vertex$.
We use Definition \ref{def:LL} to construct an ``octahedral'' gradient discretisation
$\disc=(X_{\disc,0},\nabla_\disc,\Pi_\disc)$ (see Figure \ref{fig:cevefe} for some
notations): 

\begin{enumerate}

\item The set of  geometrical entities attached to the dof is $I=\vertices$,  the set of approximation points is $S=\vertices$.
We let $\Ii=\vertices\cap \Omega$ and $\Ib=\vertices\cap \partial\Omega$,
and the partition is $\partition = \mesh$. For $U=K\in\partition$ we define $I_U = \verticescv$.

\item For $K\in\mesh$ and $\vertex\in\verticescv$, we denote by $V_{K,\vertex}$ the octahedron formed by $\x_K$, $\vertex$, and the four other vertices of $K$ that share a face of $K$ with $\vertex$. 
We then use the reconstruction \eqref{ll:piD} with the functions $(\alpha_\vertex)_{\vertex\in I_K}$ defined by
\be
  \forall \x\in K,\; \forall \vertex\in\vertices_K\,,\;\alpha_\vertex(\x) = \frac 1 3 \chi_{V_{K,\vertex}}(\x),
  \label{defpidiscddfv}
\ee
where $\chi_{V_{K,\vertex}}$ denotes the characteristic function of $V_{K,\vertex}$. This leads to
\be\label{cevefe:PiD}
\forall u\in X_{\disc,0}\,,\;\forall K\in\mesh\,,\forall\x\in K\,,\;\Pi_\disc u(\x) = \frac 1 3 \sum_{\vertex\in\verticescv} u_\vertex \chi_{V_{K,\vertex}}(\x).
\ee

\item For $K\in\mesh$ and $u\in  X_{\disc,0}$, the cell gradient is defined by
\be
\forall \x\in K,\ \gr_K u(\x) = \frac {1} {\mcv} \sum_{\edge\in\edgescv} \medge u_\edge\ncvedge\,,
\quad\mbox{ where }u_\sigma=\frac 1 3 \sum_{\vertex\in\vertices_\edge}u_\vertex
\mbox{ for all $\edge\in\edgescv$}.
\label{defgradpyddfv}\ee

\item The proof that $\Vert \nabla_\disc \cdot \Vert_{L^p(\Omega)^d}$ is a norm on $X_{\disc,0}$ is done in Lemma \ref{lem:gradddfv}.

\end{enumerate}

 \begin{figure}[!ht]
\centering
{\scalebox{.35}{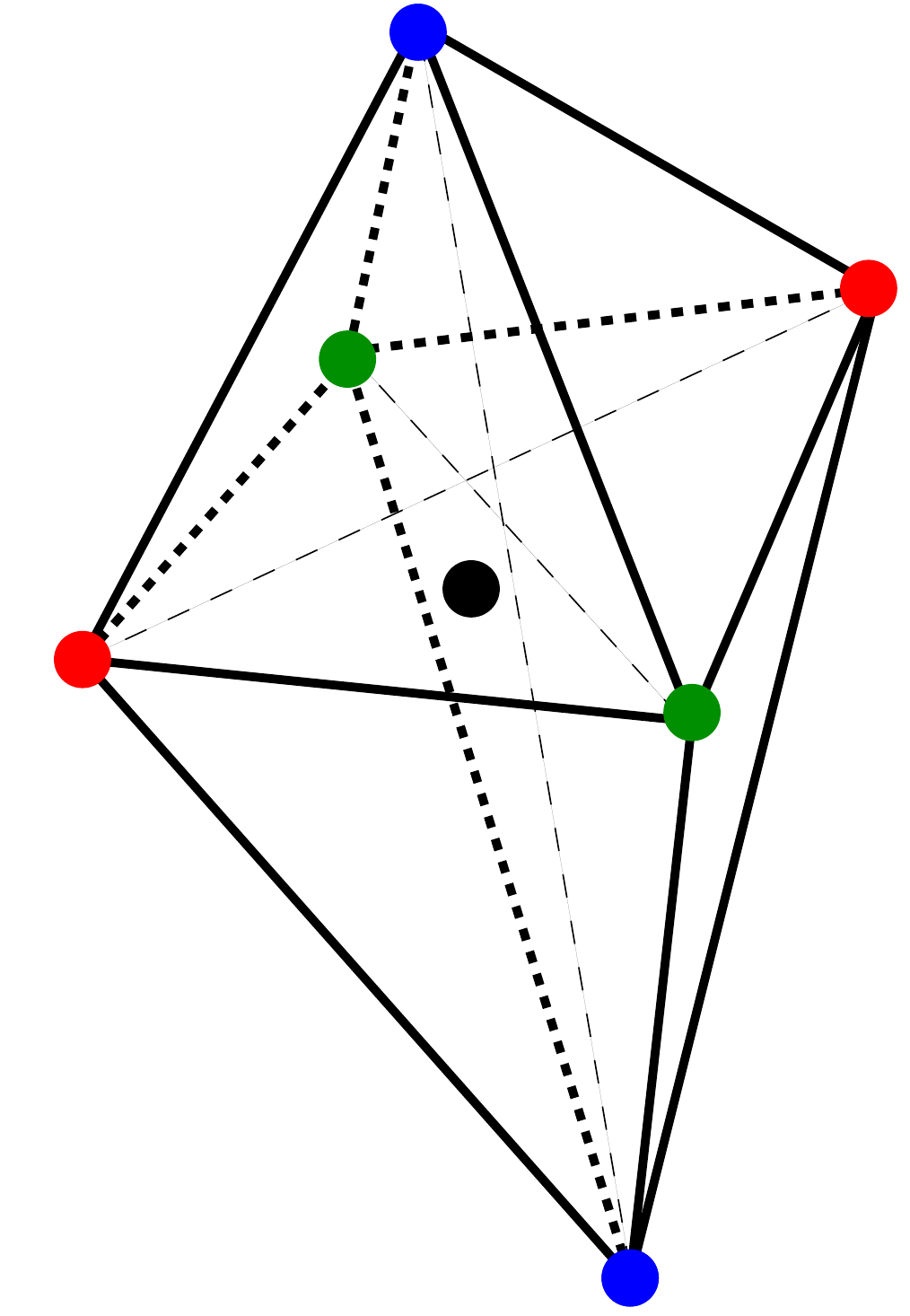}\hspace*{0.05\linewidth}\scalebox{.35}{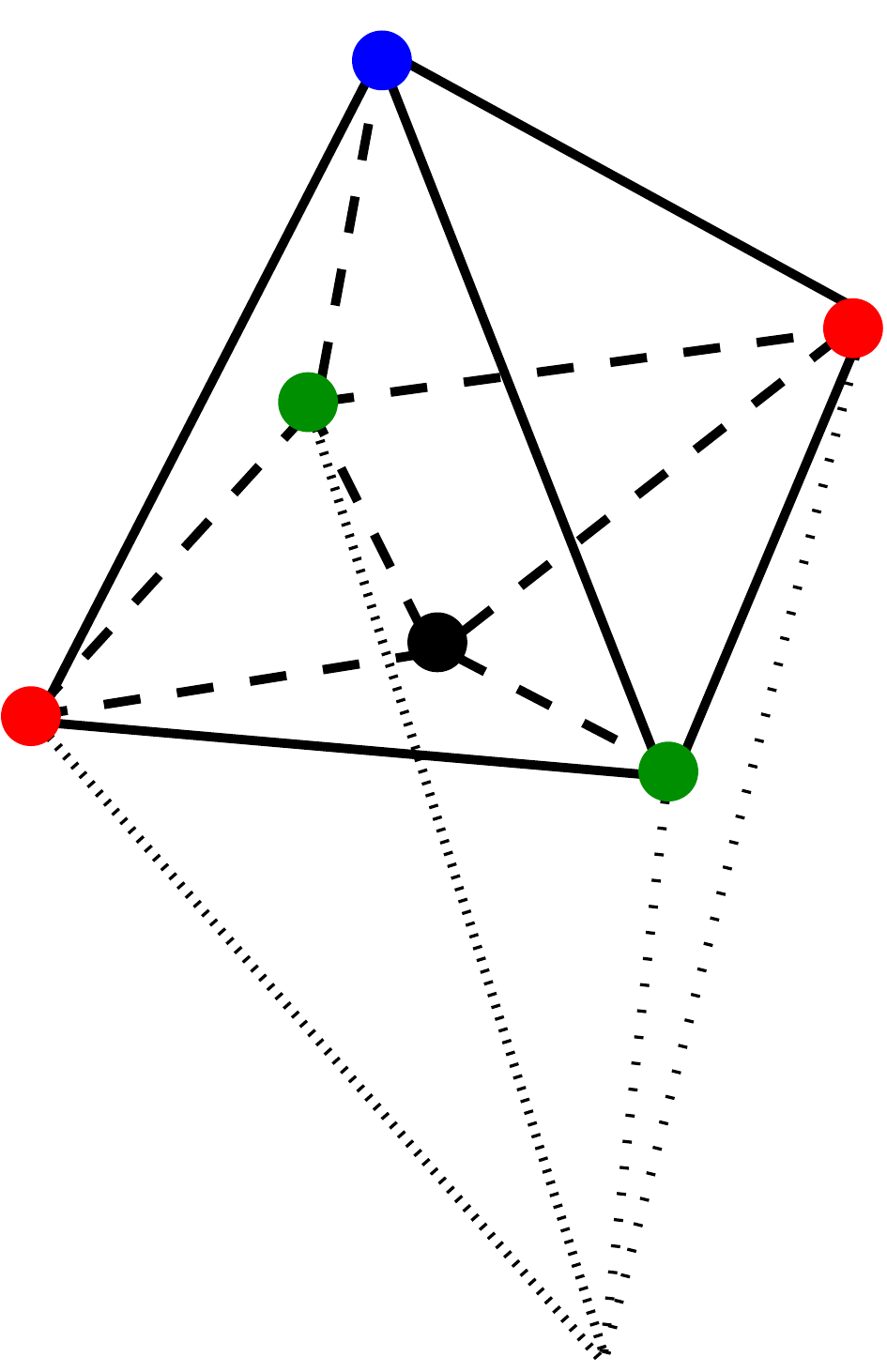}
\hspace*{0.05\linewidth}\scalebox{.35}{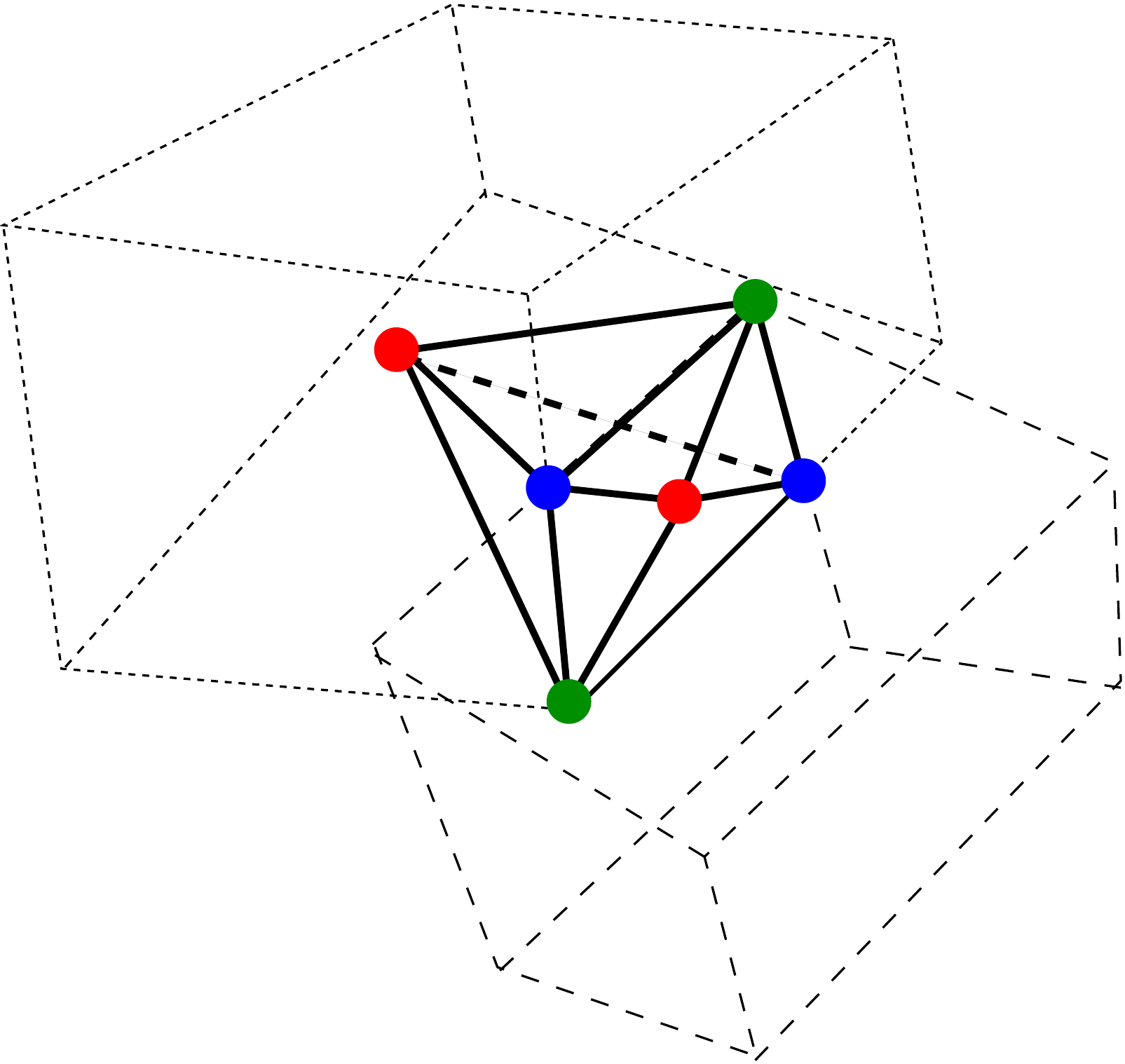}}  
\caption{\label{fig:cevefe}Left: octahedral cell $K$ for the CeVeFE-DDFV scheme.
Center: illustration of $V_{K,\vertex}$.
Right: construction of a degenerate octahedron from a non-conforming hexahedral mesh
in a heterogeneous medium ($CDEF$ is the intersection of the boundaries of two non-matching
hexahedral cells).}
\end{figure}

\begin{remark}[Octohedra and heterogeneous media]\label{ddfv:heter} 
An octahedral mesh of $\O$ can be obtained in the case where the domain $\Omega$ is the disjoint union of star-shaped hexahedra, considering the octahedra obtained from the centres of neighbourhing hexahedral cells, and the four vertices of their interface.
This also works for non-conforming hexahedral meshes, for which interfaces between cells may be different from the physical faces of the cells.

In the case of a heterogeneous media, in which the material properties (e.g.
the permeability $A$ in \eqref{pblin}) are constant inside each hexahedral cell but may be discontinuous
from one cell to the other, it is usually preferable to construct octahedral cells
that are compatible with these heterogeneities (i.e. such that the material properties are
constant inside each octahedron). This prevents from introducing a non-physical
average of $A$ in the gradient scheme \eqref{gslin}, which would lead to a loss
of accuracy of the approximate solutions. Such an octahedral
mesh can be constructed fairly easily as illustrated in Figure \ref{fig:cevefe} (right).
Each of these octahedra is built from the centre of an hexahedral cell,
the four vertices of the interface between this cell and a neighbouring
hexahedral cell, and a point selected on this interface. This interface
need not be planar.
\end{remark}

\begin{remark}[Reconstruction operator] The choice \eqref{defpidiscddfv} of $\Pi_\disc$ is driven by our desire to construct a gradient discretisation whose gradient scheme is exactly the
CeVeFE-DDFV method, for particular octahedral meshes. This choice ensures that
the discrete duality formula holds true but, as explained in Section \ref{sec:deftools}
(and as already noticed in \cite[Appendix C]{ABH13} for CeVe-DDFV methods),
it is not adapted to certain non-linear models.
% Mass-lumping is necessary to recover
%a piecewise constant reconstruction suitable for these models.
% A mass-lumped octahedral
%gradient discretisation can be obtained by considering in Definition \ref{def:pwrec}
%the set $B=\vertices$ and, for example, the Voronoi tesselation $(V_i)_{i\in B}$ generated by $\vertices$.
% We could as well take other choices for $\Pi_\disc$, e.g. the reconstruction
% such that $\Pi_\disc u = \frac{1}{6}\sum_{\vertex\in\verticescv}u_\vertex$ in $K$,
% and the corresponding gradient discretisation would still satisfy Property $(\mathcal{P})$. 
% Contrary to the choice \eqref{cevefe:PiD}, this reconstruction does not apparently lead to a discrete duality formula, which is the core of DDFV methods. 
% However, as mentioned in Section \ref{sec:deftools} it enables the use of the octahedral gradient
% discretisation on models with nonlinear reaction terms. 
% This was noticed for CeVe-DDFV methods in \cite[Appendix C]{ABH13}.
\end{remark}

The following lemma proves that the previous construction gives an LLE gradient discretisation.
It will also prove useful to show that this gradient discretisation gives back the CeVeFE-DDFV method,
and to establish Property $(\mathcal{P})$.

\begin{lemma}\label{lem:gradddfv}
Let $\disc$ be the octahedral gradient discretisation 
defined as above. For any $u\in X_{\disc,0}$ and any $K\in\mesh$,
the constant discrete gradient $\gr_K u$ is characterised by
\be\label{charddfv}
\mbox{For all opposite vertices $(\vertex,\vertex')$ of $K$, }
\gr_K u\cdot(\vertex-\vertex') = u_\vertex-u_{\vertex'}.
\ee
As a consequence,  $\Vert \nabla_\disc \cdot \Vert_{L^p(\Omega)^d}$ is a norm on $X_{\disc,0}$.

\end{lemma}

\begin{remark} The opposite vertices in the octahedra in Figure \ref{fig:cevefe} are $(A,B)$, $(C,D)$ and $(E,F)$.
\end{remark}

\begin{proof}
We first note that, since the three directions defined by the three pairs of opposite vertices in $K$ are linearly independent, \eqref{charddfv} indeed characterises one and only one vector $\gr_K u\in \R^3$. 
We therefore just have to show that the gradient defined by \eqref{defgradpyddfv} satisfies \eqref{charddfv}.
We have
\be\label{defGkddfv}
\gr_{K} u = \frac {1} {\mcv}\frac 1 3 \sum_{\vertex\in\verticescv} u_\vertex \sum_{\edge\in\edgescv|\vertex\in\vertices_\edge} \medge \ncvedge.
\ee
Let us consider for example the case where $\vertex = A$ in Figure \ref{fig:cevefe}.
For a triangular face $\sigma$ we can write $|\sigma|\n_{K,\sigma}$ as the
exterior product of two of the edges of $\sigma$ (with proper orientation).
This gives
\begin{multline*}
\sum_{\edge\in\edgescv|\vertex\in\vertices_\edge} \medge \ncvedge = \frac 1 2 (\overrightarrow{AC} \times \overrightarrow{AF} + \overrightarrow{AF} \times \overrightarrow{AD} + \overrightarrow{AD} \times \overrightarrow{AE} + \overrightarrow{AE} \times \overrightarrow{AC})\\
= \frac 1 2 (\overrightarrow{DC} \times \overrightarrow{AF}  + \overrightarrow{CD} \times \overrightarrow{AE} )
 = - \frac 1 2 \overrightarrow{CD} \times \overrightarrow{EF}.
\end{multline*}
Applying this to all vertices of $K$, and since $\mcv = \frac 1 6 \Delta_K$ with $\Delta_K = \mbox{det} (\overrightarrow{AB},\overrightarrow{CD},\overrightarrow{EF})$, we deduce
from \eqref{defGkddfv} that
\[
\gr_K u = \frac 1 {\Delta_K} \left( (u_B - u_A) \overrightarrow{CD} \times \overrightarrow{EF}  +  (u_D - u_C) \overrightarrow{EF} \times \overrightarrow{AB} +  (u_F - u_E)  \overrightarrow{AB} \times \overrightarrow{CD}\right).  
\] 
Property \eqref{charddfv} is then straightforward. Considering for example the case
$(\vertex,\vertex')=(B,A)$, the formula follows from
$ (\overrightarrow{EF} \times \overrightarrow{AB}) \cdot
\overrightarrow{AB} =(\overrightarrow{AB} \times \overrightarrow{CD})
\cdot \overrightarrow{AB}=0$ and 
$(\overrightarrow{CD} \times \overrightarrow{EF})\cdot \overrightarrow{AB}
=\mbox{det}(\overrightarrow{CD},\overrightarrow{EF},\overrightarrow{AB})=
\Delta_K$. 

\medskip

Let us now prove that $\Vert \nabla_\disc \cdot \Vert_{L^p(\Omega)^d}$ is a norm on $X_{\disc,0}$. Assume that $\Vert \nabla_\disc u \Vert_{L^p(\Omega)^d}= 0$. 
Then for any $K \in \mesh$, $\gr_K u =0$. Let us take a boundary octahedron $K$.
One of its faces $\sigma$ is entirely contained in $\partial\Omega$ and thus by
the boundary conditions all $(u_\vertex)_{\vertex\in\vertices_\edge}$ vanish.
Using \eqref{charddfv} we infer that all values of $u$ at vertices opposite to $\vertices_\edge$
also vanish. Since any vertex in $K$ either belong to $\vertices_\edge$ or is
opposite to a vertex in $\vertices_\edge$, this shows that $u$ vanishes on all vertices
of $K$. We then conclude by induction on the number of octahedra in the mesh that
$u=0$.
\end{proof}

It is now easy to detail the relationship between the octahedral gradient discretisation
$\disc$ and the CeVeFE-DDFV scheme. 

\begin{lemma}[CeVeFE-DDFV is a gradient scheme]
For any polytopal mesh $\widetilde{\polyd}$ of $\O$, there exists an octahedral mesh $\polyd$ of $\O$
such that, if $\disc$ is the octahedral gradient discretisation
defined as above from $\polyd$, then the gradient scheme \eqref{gslin} for
$\disc$ is the CeVeFE-DDFV method on $\widetilde{\polyd}$.
\end{lemma}

\begin{proof} The CeVeFE-DDFV
method on $\widetilde{\polyd}$ has cell, vertices, edge and face unknowns, and its discrete gradient is
piecewise constant on so-called ``diamond cells''. A diamond cell
is an octahedra as in Figure \ref{fig:cevefe} (left), but with $E$
chosen on the segment $[A,B]$ -- hence, the octahedra actually
degenerates into an hexahedra. The segment $[A,B]$ corresponds to an edge
of the primal mesh of the CeVeFE-DDFV method, $F$ is a point on a face
of this mesh that contains $[A,B]$, and $D$ and $C$ are points
inside the cells on each side of this face. Let us take $\polyd$
the polytopal mesh of $\O$ made of the diamond cells (degenerate
octahedra). 
It is proved in \cite[Lemma 3.1]{cou-10-dis}
that the CeVeFE-DDFV discrete gradient satisfies \eqref{charddfv}; hence,
this gradient is $\nabla_\disc$. It is then just a matter of applying the
discrete duality formula \cite[Theorem 4.1]{cou-10-dis} on the formulation
\cite[Eq. (5.4)]{cou-10-dis} of the scheme to see that the CeVeFE-DDFV
scheme for \eqref{pblin} is indeed the gradient scheme \eqref{gslin}
for $\disc$.
\end{proof}

\begin{remark}[CeVeFE-DDFV and heterogeneities]
Except on the boundary of the domain, the diamond cells of the CeVeFE-DDFV 
method are always spread on two neighbouring cells. If the primal mesh is
aligned with heterogeneities of the medium, then these heterogeneities
are actually averaged in the formulation of the CeVeFE-DDFV scheme
(see \cite[Eq. (5.2)]{cou-10-dis}). 
An option to better deal with heterogeneities is to use
the m-DDFV method \cite{boy-08-fin,CHM14}. Additional unknowns are introduced
and eliminated by writing local flux conservativity. These eliminations are easy to perform
for linear models, but require solving local nonlinear equations in the case of
nonlinear models.

As shown in Remark \ref{ddfv:heter}, even starting from
a primal hexahedral mesh aligned with the ground properties it is possible
to construct a (degenerate) octahedral mesh that is also aligned with
the heterogeneities of the medium. Using this octahedral mesh ensures a better
accuracy than the ``standard'' CeVeFE-DDFV method, at a reduced computational
cost compared to the m-DDFV method in the case of nonlinear models (no nonlinear local equation
to solve).
\end{remark}

Let us now turn to the analysis of the properties of the octahedral gradient discretisation.
The regularity of a sequence of octahedral gradient discretisations $(\disc_m)_{m\in\N}$
is merely defined as the regularity of the underlying octahedral 
meshes $(\polyd_m)_{m\in\N}$ of $\O$ in the sense of Definition \ref{def:regpolyd}. 
With this definition of regularity, the property $(\mathcal{P})$ holds for octahedral
gradient discretisations.
 
\begin{proof}[Proof of the property $(\mathcal{P})$ for octahedral gradient discretisations]
As usual, we sometimes drop the indices $m$ for sake of legibility. 
Let us define a control $\emb:X_{\disc,0} \longrightarrow X_{\polyd,0}$ of $\disc$ by $\polyd$ in the sense of Definition \ref{def:inhdfa}, where $\polyd$ is the octahedral
mesh with, for all $K\in\mesh$, the centre of $K$ defined by $\xcv = \frac 1 6 \sum_{\vertex\in\vertices_\cv} \vertex$.
We set
\be
\forall  K\in\mesh\,,\;\emb(u)_K = \frac 1 6 \sum_{\vertex\in\vertices_\cv} u_\vertex
\quad\hbox{ and }\quad \forall \edge\in\edges,\  \emb(u)_\sigma=\frac 1 3 \sum_{\vertex\in\vertices_\edge}u_\vertex.
\label{definhddfv}\ee
Let $\edge\in\edgescv$, and let us denote by $(\vertex_i)_{i=1,2,3}$ the three vertices of $\edge$ and
by $\vertex_i'$ the vertex opposite to $\vertex_i$ in $K$. We then have, for $u\in X_{\disc,0}$,
\begin{align*}
\emb(u)_\edge - \emb(u)_K &= 
\frac{1}{3}(u_{\vertex_1}+u_{\vertex_2}+u_{\vertex_3})
-\frac{1}{6}(u_{\vertex_1}+u_{\vertex_2}+u_{\vertex_3}+u_{\vertex_1'}+u_{\vertex_2'}
+u_{\vertex_3'})\\
&=\frac{1}{6}(u_{\vertex_1}-u_{\vertex_1'})
+\frac{1}{6}(u_{\vertex_2}-u_{\vertex_2'})+\frac{1}{6}(u_{\vertex_3}-u_{\vertex_3'})
=\frac 1 6 \sum_{i=1}^3 \gr_K u\cdot(\vertex_i-\vertex'_i),
\end{align*}
thanks to Lemma \ref{lem:gradddfv}. 
Therefore, since 
$\frac{|\vertex_i-\vertex'_i|}{\dcvedge}\le \frac{h_K}{\dcvedge} \le \theta_{\polyd}$,
we have $\sum_{\edge\in\edgescv}\medge\dcvedge \left|\frac {\emb(u)_\edge - \emb(u)_K} {\dcvedge}\right|^p \le \frac{1}{2^p}\theta_{\polyd}^p\, d \mcv\, |\gr_K u|^p$
and \eqref{cnt:poly1} follows.

For $K\in\mesh$ and a.e. $\x\in K$
there are three vertices $(\vertex_i)_{i=1,2,3}\subset\verticescv$  such that $\x\in V_{K,\vertex_i}$. Denoting by $\vertex_i'$ the vertex in $K$ opposite to $\vertex_i$, Lemma \ref{lem:gradddfv} again gives
\begin{align*}
 \Pi_\disc u(\x) -  \Pi_{\polyd} \emb(u)(\x)=\frac{1}{6}(u_{\vertex_1}-u_{\vertex_1'})
+\frac{1}{6}(u_{\vertex_2}-u_{\vertex_2'})+\frac{1}{6}(u_{\vertex_3}-u_{\vertex_3'})
=\frac 1 6 \sum_{i=1}^3 \gr_K u\cdot(\vertex_i-\vertex'_i).
\end{align*}
This shows that $|\Pi_\disc u(\x) - \Pi_{\disc'} \emb(u)(\x) | \le 
\frac{1}{2}h_\mesh  |\gr_K u|$ and thus that
$\Vert \Pi_\disc u-\Pi_{\polyd}\emb(u)\Vert_{L^p(\O)}\le \frac{1}{2}
h_\mesh \Vert \nabla_\disc u\Vert_{L^p(\O)^d}$.
Hence $\omega^\Pi(\disc,\polyd,\emb) \le \frac{h_{\mesh}}{2}$ and
\eqref{cnt:poly2} holds. The definition
\eqref{defgradpyddfv} implies that $\grad_\disc u = \grad_{\polyd} \emb(u)$, and therefore $\omega^\nabla(\disc,\polyd,\emb) =0$, which implies \eqref{cnt:poly3}.
By Proposition \ref{prop:inhdfa} we deduce that $(\disc_m)_{m\in\N}$ is coercive in the sense of Definition \ref{def-coer}, limit-conforming in the sense of Definition \ref{def-limconf}, and compact in the sense of  Definition \ref{def-comp}.

The bound on $\theta_\polyd$ forces the three vectors $(\vertex-\vertex')_{(\vertex,
\vertex')\mbox{\scriptsize~opposite in $K$}}$ to be of similar length
and ``really non-coplanar'' -- that is to say with a determinant of order the cube of
their similar length -- which gives an estimate on $||\gr_K||_\infty$. 
Hence, the regularity factor $\regLLE(\disc_m)$ is bounded and
Lemma \ref{lem:gradddfv} proves the consistency. \end{proof}

\subsection{Hybrid mimetic mixed schemes}\label{sec:hmm}

\subsubsection{Fully hybrid scheme}\label{sec:hyb}

Since the 50's, several schemes have been developed with the objective to
satisfy some form of calculus formula at the discrete level. These schemes
are called mimetic finite difference (MFD) or compatible discrete operator
(CDO) schemes. Contrary to DDFV methods that design a discrete operators
and duality products to satisfy fully discrete calculus formula, MFD/CDO methods design discrete operators
that satisfy a Stokes formula that involves both continuous and discrete
functions.
Depending on the choice of the location of the main geometrical entities attached to the dof (faces or vertices),
two different MFD/CDO families exist.
We refer to \cite{mfdrev} for a review on MFD methods, and to \cite{EB14,phdbonelle}
(and reference therein) for CDO methods.

A first MFD method, hereafter called hybrid MFD (hMFD),
is designed by using the fluxes through the mesh faces
as initial unknowns. This requires to recast \eqref{pblin} in a mixed form,
i.e. to write $\overline{q}=A\nabla\bu$ and $\div(\overline{q})=f$,
and to discretise this set of two equations. The resulting scheme takes a form
that is apparently far from the gradient scheme \eqref{gslin}.
It was however proved in \cite{dro-10-uni} that this hMFD can be actually
embedded in a slightly larger family that also contains hybrid finite volume (HFV)
methods \cite{sushi} and mixed finite volume (MFV) methods \cite{dro-06-mix}.
This family has been called hybrid mimetic mixed (HMM) schemes; each scheme
in this family can be written in three different ways, depending on the
considered approach (hMFD, HFV or MFV). The HFV formulation of an HMM
scheme is very close to the weak formulation \eqref{pblinw} of the elliptic PDE;
it actually consists in writing this weak formulation with a discrete gradient
and a stabilisation term (bilinear form on $(u,v)$). It was proved in
\cite{dro-12-gra} that the discrete gradient can be modified to include the
stabilisation terms, and thus that all HMM methods -- which means all
hMFD methods also -- are actually gradient schemes.

The discrete elements that define an HMM gradient discretisation are the following.
We again refer to Definition \ref{def:LL} for the construction of $\disc$.

\begin{enumerate}

\item 
Let $\polyd$ be a polytopal mesh of $\O$ as in Definition \ref{def:polymesh}.
The  geometrical entities attached to the dof are $I =\mesh\cup\edges$ and the approximation 
points are $S=((\x_K)_{K\in\mesh},(\centeredge)_{\edge\in\edges})$.
We let $\Ii=\mesh\cup\edgesint$ and $\Ib=\edgesext$, and
we have $X_{\disc,0} = X_{\polyd,0}$ as defined by \eqref{dfa:spaceH}.
The partition is  $\partition = \{D_{K,\sigma}: K\in\mesh, \sigma\in\edgescv\}$
and, for $U=D_{K,\edge}\in\partition$, we set $I_U = \{K\}\cup \edgescv$.

\item The reconstruction \eqref{ll:piD} is defined by the functions
$\alpha_K\equiv 1$ in $K$ and $\alpha_K\equiv 0$ outside $K$, for all $K\in\mesh$, and $\alpha_{\edge}\equiv 0$ for $\edge \in \edgescv$.
Recalling the definition \eqref{def:pipolyd} we therefore have
 \[
  \forall v\in X_{\disc,0},\ \forall K\in\mesh,\ \forall \x\in K,\  \Pi_\disc v(\x) = 
\Pi_\polyd v(\x)=v_K.
 \]

\item Recalling the definition of the polytopal gradient \eqref{def:nablapolyd}, we start the construction of the discrete gradient by setting
\begin{equation}
\forall v\in X_{\disc,0}\,,\;\forall K\in\mesh\,,\;
\nabla_K v=(\nabla_\polyd v)_{|K}=\frac{1}{|K|}\sum_{\edge\in\edgescv}|\sigma|v_\sigma\n_{K,\edge}.
\label{def:gradK}\end{equation}                
This gradient is linearly exact, but it does not lead to the norm property:
it can vanish everywhere even for $v\not=0$ (it suffices to take $v_\edge = 0$ for all $\edge\in\edges$). 
We therefore add a stabilisation that is constant in each half-diamond,
and that vanishes on interpolations of affine functions:
\be\label{hmm:nablaD}
\forall v\in X_{\disc,0}\,,\;\forall D_{K,\sigma}\in\partition\,,\; \forall \x\in D_{K,\edge}\,,\;
\gr_{D_{K,\edge}} v(\x) = \nabla_K v + \frac{\sqrt{d}}{d_{K,\edge}}[\iso_K \tpen_K(\incr_K(v))]_\edge \bfn_{K,\edge},
\ee
where 
\begin{itemize}
\item $\incr_K(v)=(v_\edge-v_K)_{\edge\in\edgescv}$,
\item $\tpen_K:\R^\edgescv\mapsto \R^\edgescv$ is
the linear mapping defined by $\tpen_K(\xi)=(\tpen_{K,\edge}(\xi))_{\edge\in\edgescv}$ with
\[
\tpen_{K,\edge}(\xi)=\xi_\sigma - \left(\frac{1}{|K|}\sum_{\edge'\in\edgescv}|\edge'|\xi_{\edge'} \bfn_{K,\edge'}\right)\cdot(\centeredge-\x_K),
\]
\item $\iso_K$ is an isomorphism of the vector space ${\rm Im}(\tpen_K)\subset \R^\edgescv$.
\end{itemize}
 \item The norm property of $\Vert \nabla_\disc \cdot \Vert_{L^p(\Omega)^d} $ on $X_{\disc,0}$ follows from the estimate (see \cite[Lemma 5.3]{dro-12-gra})
\begin{equation}
  \Vert u \Vert_{\polyd,0,p}   \le
\cter{cte:hmm} \Vert \nabla_\disc u\Vert_{L^p(\Omega)^d}, 
\label{cnt1:hmm}
\end{equation}
where $\ctel{cte:hmm}\ge 0$ depends on an upper bound on the regularity factor $\theta_\polyd$ of the polytopal mesh and on the regularity factor $\zeta_\disc$.
This last factor is defined as the smallest number such that, for all $K\in\mesh$ and all
$\xi\in \R^\edgescv$,
\[
\zeta_\disc^{-1}\sum_{\edge\in\edgescv}|D_{K,\edge}|\left|\frac{\tpen_{K,\edge}(\xi)}{d_{K,\edge}}\right|^p
\le \sum_{\edge\in\edgescv} |D_{K,\edge}|\left|\frac{[\iso_K\tpen_K(\xi)]_\edge}{d_{K,\edge}}\right|^p
\le\zeta_\disc\sum_{\edge\in\edgescv}|D_{K,\edge}|\left|\frac{\tpen_{K,\edge}(\xi)}{d_{K,\edge}}\right|^p.
\]
\end{enumerate}

\begin{remark}
The face degree of freedom $v_\edge$ corresponds to the hybridisation of
the hMFD methods. 
\end{remark}

The freedom of choice of the isomorphisms $(\iso_K)_{K\in\mesh}$ ensures that
all hMFD, HFV and MFV schemes are covered by the framework (there are several
such schemes, due to different possible stabilisation parameters).
More precisely, \cite{dro-12-gra} proves that for any HMM scheme $\mathcal S$ on
$\polyd$ there exists a family of isomorphism $(\iso_K)_{K\in\mesh}$ such that,
if $\disc$ is defined as above, then the gradient scheme \eqref{gslin} is $\mathcal S$.

The proof of the property $(\mathcal{P})$ for HMM methods was originally given in \cite{dro-12-gra}.
We show here how the notion of gradient discretisations controlled by polytopal toolboxes notably
simplifies this proof. We say that a sequence of HMM gradient discretisations
$(\disc_m)_{m\in\N}$ is regular if the sequence of underlying polytopal meshes $(\polyd_m)_{m\in\N}$
is regular in the sense of Definition \ref{def:regpolyd}, and if $(\zeta_{\disc_m})_{m\in \N}$ is bounded.

\begin{proof}[Proof of the property $(\mathcal{P})$ for HMM gradient discretisations]
Let $(\disc_m)_{m\in\N}$ be HMM gradient discretisations built on polytopal meshes $(\polyd_m)_{m\in\N}$, and let us 
define a control of $\disc_m$ by $\polyd_m$ in the sense of Definition \ref{def:inhdfa}. 
We drop the index $m$ from time to time.
Since $X_{\disc,0}=X_{\polyd,0}$, we can take $\emb={\rm Id}$.
Estimate \eqref{cnt:poly1} is given by \eqref{cnt1:hmm}.
Relation \eqref{cnt:poly2} follows immediately since $\omega^\Pi(\disc,\polyd,\emb)=0$,
owing to $\Pi_{\disc} u = \Pi_{\polyd}u=\Pi_{\polyd}\emb(u)$. 
Recalling that $|D_{K,\edge}|=
\frac{|\edge|d_{K,\edge}}{d}$ we have 
\be \label{HMM:average.grad}
 \int_K \grad_\disc u(\x)\d\x  = \mcv \nabla_K u +\frac{1}{\sqrt{d}}
\sum_{\edge\in\edgescv} |\edge|[\iso_K \tpen_K(\incr_K(u))]_\edge \bfn_{K,\edge}.
\ee
The definition of $\tpen_K$ and the property $\sum_{\edge\in\edgescv}|\edge|\bfn_{K,\edge}
(\centeredge-\x_K)^T=|K|{\rm Id}$ (a consequence of Stokes' formula) show that
for any $\eta\in {\rm Im}(\tpen_K)$ we have
$\sum_{\edge\in\edgescv}|\edge| \eta_\edge  \bfn_{K,\edge}=0$.
Hence, since ${\rm Im}(\iso_K)={\rm Im}(\tpen_K)$, \eqref{HMM:average.grad} gives
\[
  \int_K \grad_\disc u(\x)\d\x  = \mcv \nabla_K u =\mcv \nabla_\polyd\emb(u)_{|K},
\]
which shows that $\omega^\nabla(\disc,\polyd,\emb)=0$, and thus
that \eqref{cnt:poly3} holds. The coercivity, limit-conformity and compactness
of $(\disc_m)_{m\in\N}$ therefore follow from Proposition \ref{prop:inhdfa}.
Since HMM gradient discretisations are LLE gradient discretisations, the
consistency of $(\disc_m)_{m\in\N}$ readily follows from Proposition \ref{prop:LL},
after noticing that the regularity assumption on $(\disc_m)_{m\in\N}$
gives a bound on $(\regLLE(\disc_m))_{m\in\N}$. \end{proof}

\subsubsection{The SUSHI scheme} 

With the notations of Section \ref{sec:prop}, the SUSHI scheme \cite{sushi} is the gradient scheme obtained by a barycentric condensation of the HMM gradient discretisation.
For simplicity, we only consider here the case  when all face unknowns are eliminated (the ``SUCCES'' version in \cite{sushi}),  although more accurate methods could be used in the case of coarse meshes in heterogeneous domains.
 
\begin{enumerate}
\item Let $\polyd$ be a polytopal toolbox of $\O$ in the sense of
Definition \ref{def:polydis}. 
We first define an HMM gradient discretisation  $\disc = (X_{{\disc},0},\nabla_{{\disc}},\Pi_{{\disc}})$, as in the section above, for which $I = \mesh\cup\edges$.
\item We introduce $\bcI = \mesh\cup\edgesext$, and for all $\edge\in\edgesint$ we select $H_\edge\subset \bcI$ and introduce barycentric coefficients $\beta^\edge_i$ such that
\[
 \sum_{i\in H_\edge}\beta^\edge_i = 1\quad\hbox{ and }\quad\centeredge = \sum_{i\in H_\edge}\beta^\edge_i \x_i,
\]
which corresponds to \eqref{barcomb}.
\item The SUSHI gradient discretisation is the corresponding barycentric
condensation $\bcdisc$ of ${\disc}$ in the sense of Definition \ref{def:bcgd}.

\item We have $\Pi_\bcdisc=\Pi_{{\disc}}$ since this reconstruction
is only built from the values at the centres of the cells.

\end{enumerate}

% 
% 
% \begin{definition}[Regularity of SUSHI gradient discretisations]
% A sequence $(\disc_m)_{m\in\N}$ of SUSHI gradient discretisations is regular if the sequence $(\polyd_m)_{m\in\N}$ is regular in the sense of Definition \ref{def:regpolyd}
% and if $(\regbc(\disc_m)+\zeta_{\overline{\disc}_m})_{m\in\N}$ is bounded,
% where $\regbc$ is defined in Definition \ref{def:bcgd} and 
% $\zeta_{\disc}$ is defined in Definition \ref{def:reghmm} for the HMM gradient discretisations $(\overline{\disc}_m)_{m\in\N}$.
% \end{definition}

The property $(\mathcal{P})$ for this barycentric condensation of the HMM method
is a consequence of the property $(\mathcal{P})$ for the HMM method and
of Theorem \ref{th:bcgd}, assuming that $(\regbc(\bcdiscm)+\zeta_{\disc_m}+ \theta_{\polyd_m})_{m\in\N}$ is bounded.

\subsection{Nodal mimetic finite difference methods}\label{sec:nMFD}

The nodal MFD method (nMFD) is described in \cite{BBL09}. 
We present here a gradient discretisation which enables us to write the nMFD method for \eqref{pblin} as the gradient scheme \eqref{gslin}.
Let $\polyd$ be a polytopal mesh of $\O$ in the sense of Definition \ref{def:polymesh}.
For each $K\in\mesh$ we choose non-negative weights $(\omega_K^\vertex)_{\vertex\in
\vertices_K}$ such that the quadrature
\be\label{quad:K}
\int_K w(\x)\d \x\approx \sum_{\vertex\in\vertices_K}\omega_K^\vertex w(\vertex)
\ee
is exact for constant functions $w$, which means that $\sum_{\vertex\in\vertices_K} \omega_K^\vertex=|K|$.
For each face $\edge\in\edgescv\cap \edgesint$, we also choose non-negative weights
$(\omega_\edge^\vertex)_{\vertex\in\vertices_\edge}$ such that the quadrature
\be\label{quad:edge}
\int_\edge w(\x)\d s(\x)\approx \sum_{\vertex\in\vertices_\edge}\omega_\edge^\vertex 
w(\vertex)
\ee
is exact for affine functions $w$. This is equivalent to
$\sum_{\vertex\in\vertices_\edge}\omega_\edge^\vertex =|\edge|$
and $\sum_{\vertex\in\vertices_\edge}\omega_\edge^\vertex \vertex = |\edge|\centeredge$.
We will assume the following property on these weights. This property is not
required in the construction of the nMFD, but it is used
to identify the nMFD with a gradient scheme. We note that this assumption
is not very restrictive, since it holds for
any natural choice of weights for \eqref{quad:edge}.
\be\label{assum:weights}
\forall K\in\mesh\,,\;
\forall \vertex\in\vertices_K\,,\;\exists \sigma\in\edges_{K,\vertex}\mbox{ such that } 
\omega_\edge^\vertex\not=0,
\ee
where $\edges_{K,\vertex}=\{\edge\in\edgescv\,:\,\vertex\in\vertices_\edge\}$ is the set of faces of $K$ that have $\vertex$ as one
of their vertices. 
For each cell $K \in \mesh$, we define its centre as 
\be\label{def:xtK}
{\x}_K=\frac{1}{|K|}\sum_{\vertex\in\vertices_K}\omega_K^\vertex \vertex
\ee
and we select a partition $(V_{K,\vertex})_{\vertex\in\vertices_K}$ such that
\be\label{mes:part}
\forall \vertex\in\vertices_K\,,\;|V_{K,\vertex}|=\sum_{\edge\in\edges_{K,\vertex}}
\omega_\edge^\vertex \frac{|D_{K,\edge}|}{|\edge|}=\frac{1}{d}\sum_{\edge\in\edges_{K,\vertex}}
\omega_\edge^\vertex d_{K,\edge}.
\ee

We can then again use  Definition \ref{def:LL} to construct the gradient discretisation $\disc = (X_{\disc,0},\nabla_\disc,\Pi_\disc)$ corresponding to the nMFD scheme.

\begin{enumerate}

\item The set of geometrical entities attached to the dof is $I =\vertices$ and  the set of approximation points is $S=I$. We set $\Ii=\vertices\cap\Omega$ and
$\Ib=\vertices\cap\partial\Omega$.
The partition is $\partition = (V_{K,\vertex})_{K\in\mesh,\,\vertex\in\vertices_K}$, and for $U = V_{K,\vertex}$ we let $I_U = \verticescv$.

\item For $U = V_{K,\vertex}$ we choose in the reconstruction \eqref{ll:piD} the functions
\be\label{nMFD:piD}
\forall \x\in U\,,\;\forall \vertex'\in\vertices_K\,,\;
\alpha_{\vertex'}(\x) :=\frac{1}{|K|}\omega_K^{\vertex'}.
\ee
This leads to
 \be\label{nMFD:piDb}
  \forall v\in X_{\disc,0},\ \forall K\in\mesh,\ \forall \x\in K,\  \Pi_\disc v(\x) = v_K :=\frac{1}{|K|}\sum_{\vertex\in
\vertices_K}\omega_K^\vertex v_\vertex.
 \ee

\item In a similar way as for the HMM method, the reconstructed gradient is
the sum of a constant gradient in each cell and of stabilisation terms in each
$V_{K,\vertex}$. We set
\be\label{nMFD:nablaK}
\forall v\in X_{\disc,0}\,,\;\forall K\in\mesh\,,\;
\nabla_K v = \frac{1}{|K|}\sum_{\edge\in\edgescv} \left(\sum_{\vertex\in\vertices_\edge}\omega_\edge^\vertex v_\vertex\right)\n_{K,\edge}.
\ee
and
\be\label{nMFD:nablaD}
\forall v\in X_{\disc,0}\,,\;\forall V_{K,\vertex}\in\partition\,,\;\forall \x\in V_{K,\vertex}\,,\;
\gr_{V_{K,\vertex}} v(\x)= \nabla_K v+ \frac{1}{h_K}
[\iso_K \tpen_K(\incr_K(v))]_\vertex \nv_{K,\vertex}
\ee
where
\begin{itemize}
\item $\nv_{K,\vertex}=\frac{h_K}{d|V_{K,\vertex}|}\sum_{\edge\in\edges_{K,\vertex}}
\omega_\edge^\vertex \n_{K,\edge}$,
\item $\incr_K(v)=(v_\vertex-v_K)_{\vertex\in\vertices_K}$ with $v_K=\frac{1}{|K|}\sum_{\vertex\in
\vertices_K}\omega_K^\vertex v_\vertex$ as in \eqref{nMFD:piDb},
\item $\tpen_K:\R^{\vertices_K}\mapsto \R^{\vertices_K}$ is the linear mapping described
by $\tpen_K(\xi)=(\tpen_{K,\vertex}(\xi))_{\vertex\in\vertices_K}$ with
\be\label{nMFD:tpen}
\tpen_{K,\vertex}(\xi)=\xi_\vertex-\nabla_K \xi\cdot (\vertex- {\x}_K),
\ee
where $\nabla_K \xi$ is defined as in \eqref{nMFD:nablaK}, and ${\x}_K$ is the centre of $K$  defined by \eqref{def:xtK},
\item $\iso_K$ is an isomorphism of the space ${\rm Im}(\tpen_K)\subset \R^{\vertices_K}$.
\end{itemize}
 
\item The proof that $\Vert \nabla_\disc \cdot \Vert_{L^p(\Omega)^d}$ is a norm on $X_{\disc,0}$ is the consequence of the following inequality, which is obtained in the same way as in \cite[Lemma 5.3]{dro-12-gra}: defining $u_K=\frac{1}{|K|}\sum_{\vertex\in
\vertices_K}\omega_K^\vertex u_\vertex$ (see \eqref{nMFD:piDb}), we have
\be
\label{controlvvvK}
\sum_{\vertex\in\vertices_K} |V_{K,\vertex}|\, \left|\frac{u_\vertex-u_K}{h_K}\right|^p\le \cter{cte:nmfd} ||\nabla_\disc u||_{L^p(K)^d}^p,
\ee
with $\ctel{cte:nmfd}$ only depending on an upper bound on $\theta_\polyd$ and on the regularity $\zeta_\disc$. This factor is defined as the smallest number such that, for all $K\in\mesh$ and all $\xi\in \R^{\vertices_K}$, 
\be\label{hyp:iso:nodal}
\zeta_\disc^{-1}
\sum_{\vertex\in\vertices_K} |V_{K,\vertex}|\,
\left|\frac{\tpen_{K,\vertex}(\xi)}{h_K}\right|^p
\le
\sum_{\vertex\in\vertices_K} |V_{K,\vertex}|\,
\left|\frac{[\iso_K\tpen_K(\xi)]_\vertex}{h_K}\right|^p
\le
\zeta_\disc
\sum_{\vertex\in\vertices_K} |V_{K,\vertex}|\,
\left|\frac{\tpen_{K,\vertex}(\xi)}{h_K}\right|^p.
\ee
\end{enumerate}

Under Assumption \eqref{assum:weights} it is proved in \cite{koala} that the gradient scheme \eqref{gslin} obtained from this gradient discretisation is identical to the
nMFD method of \cite{BBL09} for \eqref{pblin}.

%\begin{remark} If \eqref{quad:K} is exact at order 1 (instead of order 0), then $\widetilde{\x}_K$
%defined by \eqref{def:xtK} is the centre of gravity of $K$.
%\end{remark}

\begin{remark} The second equality in \eqref{mes:part}
comes from $|D_{K,\edge}|=\frac{|\edge|d_{K,\edge}}{d}$,
and this choice of $|V_{K,\vertex}|$ is compatible with the requirement that $\sum_{\vertex\in\vertices_K}|V_{K,\vertex}|=|K|$.
The detailed construction and geometric properties of $V_{K,\vertex}$
are not needed for the analysis of the method or for its
implementation. The only required information is the measure of this set.

Other choices of $V_{K,\vertex}$ are possible. For example, we could take
all $(V_{K,\vertex})_{\vertex\in\vertices_K}$ of the same measure
$\frac{|K|}{{\rm Card}(\vertices_K)}$, and Property $(\mathcal{P})$ would still
be valid. However, a stronger assumption than \eqref{assum:weights} would be required to ensure the coercivity of the corresponding gradient discretisations;
we would need $\sum_{\edge\in\edges_{K,\vertex}}\omega_\edge^\vertex\ge c h_K^{d-1}$
with $c>0$ not depending on $K$ or $\vertex$.
\end{remark}

% \begin{definition}[Regularity of the nMFD gradient discretisation]\label{def:regnmfd}
% If $\disc$ is an nMFD gradient discretisation as above, we define $\zeta_\disc$
% as the smallest number such that, for all $K\in\mesh$ and all $\xi\in \R^{\vertices_K}$,
% \be\label{hyp:iso:nodal}
% \zeta_\disc^{-1}
% \sum_{\vertex\in\vertices_K} |V_{K,\vertex}|\,
% \left|\frac{\tpen_{K,\vertex}(\xi)}{h_K}\right|^p
% \le
% \sum_{\vertex\in\vertices_K} |V_{K,\vertex}|\,
% \left|\frac{[\iso_K\tpen_K(\xi)]_\vertex}{h_K}\right|^p
% \le
% \zeta_\disc
% \sum_{\vertex\in\vertices_K} |V_{K,\vertex}|\,
% \left|\frac{\tpen_{K,\vertex}(\xi)}{h_K}\right|^p.
% \ee
% A sequence $(\disc_m)_{m\in\N}$ of nMFD gradient discretisations
% is regular if the sequence $(\polyd_m)_{m\in\N}$ is regular in the sense of Definition \ref{def:regpolyd} and if $(\zeta_{\disc_m})_{m\in\N}$ is bounded.
% \end{definition}

\begin{remark} Contrary to the HMM gradient discretisation, the nMFD gradient discretisation
does not have a piecewise constant reconstruction for the natural choice of unknowns, nor for any obvious
choice of unknowns. 
It should therefore be modified, e.g. by mass-lumping as in Section \ref{sec:masslump}, to be applicable in practice to certain non-linear models.
\end{remark}

\begin{remark} In the case of octahedral meshes, if the stabilisation term
in $\nabla_\disc$ is set to $0$, then the discrete space and gradient of
the nMFD gradient discretisation are identical to the discrete space and gradient of the
octahedral gradient discretisation (and thus of the CeVeFE-DDFV
method on degenerate octahedra). The only difference remains in the definition of
the reconstruction $\Pi_\disc$.
\end{remark}

\begin{remark} In \cite{EB14}, the authors construct a vertex-based compatible discrete operator scheme; they
show that it belongs to the nMFD family of schemes \cite[Section 3.5]{EB14} and that it satisfies the coercivity, consistency and limit-conformity properties \cite[Section 4.5]{EB14}.
\end{remark}

The regularity of sequences of nMFD gradient discretisations $(\disc_m)_{m\in\N}$ is defined as the
regularity of the underlying polytopal meshes $(\polyd_m)_{m\in\N}$
(see Definition \ref{def:regpolyd}) and the boundedness of $(\zeta_{\disc_m})_{m\in \N}$.

\begin{proof}[Proof of the property $(\mathcal{P})$ for the nMFD gradient discretisation]
As in previous proofs, we drop indices $m$ from time to time.
We define a control $\emb$ of $\disc$ by $\polyd$, in the sense of Definition \ref{def:inhdfa},
by
\be
\forall  K\in\mesh,\ \emb(u)_K = u_K=\frac{1}{|K|}\sum_{\vertex\in\vertices_K}\omega_K^\vertex u_\vertex \quad\hbox{ and } \quad\forall \edge\in\edges,\ \emb(u)_\sigma=\frac 1 {\medge} \sum_{\vertex\in\vertices_\edge}\omega_\edge^\vertex u_\vertex.
\label{definhnmfd}\ee
Let us prove \eqref{cnt:poly1}.
Since $\sum_{\vertex\in\vertices_\edge}\omega_\edge^\vertex=|\edge|$
we have
$ \emb(u)_\edge - \emb(u)_K = \frac 1 {\medge} \sum_{\vertex\in\vertices_\edge}\omega_\edge^\vertex (u_\vertex - u_K)$.
Therefore, using Jensen's inequality and the fact that $\frac{1}{\dcvedge}
\le \frac{\theta_{\polyd}}{h_K}$
we find
\begin{multline*}
  \sum_{\edge\in\edgescv}\!\medge\dcvedge\!\left|\frac {\emb(u)_\edge - \emb(u)_K}{\dcvedge}\right|^p  \le \sum_{\edge\in\edgescv}\!\dcvedge \sum_{\vertex\in\vertices_\edge}\! \omega_\edge^\vertex \left|\frac {u_\vertex - u_K}{\dcvedge}\right|^p\le \theta_{\polyd}^{p} \sum_{\edge\in\edgescv}\dcvedge \sum_{\vertex\in\vertices_\edge}\omega_\edge^\vertex \!\left|\frac {u_\vertex - u_K}{h_K}\right|^p\\
    \le \theta_{\polyd}^{p}\sum_{\vertex\in\vertices_K} \left(\sum_{\edge\in\edges_{\cv,\vertex}}\dcvedge\omega_\edge^\vertex\right)\left|\frac {u_\vertex - u_K}{h_K}\right|^p =\theta_{\polyd}^{p} d \sum_{\vertex\in\vertices_K} |V_{K,\vertex}|\, \left|\frac{u_\vertex-u_K}{h_K}\right|^p.
\end{multline*}
We conclude the proof of \eqref{cnt:poly1} thanks to \eqref{controlvvvK}.
Since $\Pi_{\disc} u = \Pi_{\polyd}\emb(u)$, we have $\omega^\Pi(\disc,\polyd,\emb)=0$
and \eqref{cnt:poly2} follows.
For $K\in\mesh$ we have $\grad_K u = (\grad_\polyd\emb(u))_{|K}$.
Therefore
\be\label{nMFD:id.grad}
 \int_K \grad_\disc u(\x)\d\x  = \mcv (\nabla_\polyd \emb(u))_{|K} +
\frac{1}{d}\sum_{\vertex\in\verticescv}  [\iso_K \tpen_K(\incr_K(v))]_\vertex
\sum_{\edge\in\edges_{K,\vertex}} \omega_\edge^\vertex \n_{K,\edge}.
\ee
Similarly as for the HMM method, for any $\eta\in {\rm Im}(\tpen_K)$
we have $\sum_{\vertex\in\verticescv} \eta_\vertex\sum_{\edge\in\edges_{K,\vertex}} \omega_\edge^\vertex \n_{K,\edge}=0$.
Hence, the last term in \eqref{nMFD:id.grad} vanishes and \eqref{cnt:poly3} holds since $\omega^\nabla(\disc,\polyd,\emb)=0$.
Hence the hypotheses of Proposition \ref{prop:inhdfa} are verified, which shows that $(\disc_m)_{m\in\N}$ is coercive, limit-conforming
and compact.

By noticing that $\regLLE(\disc_m)$ remains bounded by regularity assumption on $(\disc_m)_{m\in\N}$,
the consistency of $(\disc_m)_{m\in\N}$ is an immediate consequence of Proposition \ref{prop:LL}
since nMFD gradient discretisations are LLE gradient discretisations. \end{proof}

\subsection{Vertex approximate gradient (VAG) methods}

Successive versions of the VAG schemes have been
described in several papers \cite{eym-12-sma,VAG-compGeo}. 
VAG methods stem from the idea that it is often computationally 
efficient to have all unknowns located at the vertices of the mesh,
especially with tetrahedral meshes (which have much less vertices than cells).
It is however known that schemes with degrees of freedom at the vertices
may lead to unacceptable results for the transport of a species in a heterogeneous domain,
in particular for coarse meshes (one layer of mesh for one homogeneous layer, for example).
The VAG schemes are an answer to this conundrum. After all possible local eliminations,
the VAG schemes only has vertex unknowns, and it has been shown to cure the numerical issues
for coarse meshes and heterogeneous media \cite{VAG-compGeo,VAGcanum,zamm2013};
this is due to a specific mass-lumping that spreads the reconstructed
function between the centre of the control volumes and the vertices.
Let us remark that the original version of the VAG scheme in  \cite{eym-12-sma}
uses the same nodal formalism as Section \ref{sec:nMFD}, but has been shown in the FVCA6
3D Benchmark \cite{3Dbench} to be less precise than the version presented here.

The VAG scheme is defined as a barycentric condensation and mass-lumping of
the $\P_1$ gradient discretisation on a sub-tetrahedral mesh.
\begin{enumerate}
\item Let $\polyd=(\mesh,\edges,\centers,\vertices)$ be a polytopal mesh of $\O$ in the sense of
Definition \ref{def:polymesh}, except the hypothesis that the faces $\edge\in\edges$ are planar which is not necessary here. We define a tetrahedral mesh by the following procedure.
For any $K\in\mesh$, $\edge\in\edgescv$, and $\vertex,\vertex'\in \vertices_\edge$ such that $[\vertex,\vertex']$ is an edge of $\edge$,
we define the tetrahedron $T_{K,\edge,\vertex,\vertex'}$ by its four vertices $\xcv,\x_\edge,\vertex,\vertex'$ (see Figure \ref{fig:vag}), where the
point $\x_\edge$ corresponding to the face $\edge$ is
\be\label{def:xedgeVAG}
\x_\edge = \frac 1 {{\rm Card}(\vertices_\edge)} \sum_{\vertex\in\vertices_\edge} \vertex.
\ee
We denote by $\polyd^T$ the simplicial mesh corresponding to
these $T_{K,\edge,\vertex,\vertex'}$.

\item We let $\overline{\disc} = (X_{\overline{\disc},0},\nabla_{\overline{\disc}},\Pi_{\overline{\disc}})$ be the $\P_1$ gradient discretisation defined from $\polyd^T$
as in Section \ref{sec:Pk} for $k=1$. 
For the gradient discretisation $\overline{\disc}$, the set $I$ of geometrical entities attached to the dof is $I=\mesh\cup\vertices\cup\edges$, and the set of $S$ of approximation points of
is  $S=((\x_K)_{K\in\mesh},(\vertex)_{\vertex\in\vertices},(\x_\edge)_{\edge\in\edges})$.
We define $\obcdisc$ as the barycentric condensation of $\overline{\disc}$ (see Definition \ref{def:bcgd}) such that $\bcI=\mesh\cup\vertices$ and the degrees of freedom attached to $\edges$ are eliminated by setting $H_\edge = \vertices_\edge$ and the coefficients
$\beta^\edge_\vertex = 1/{\rm Card}(\vertices_\edge)$ for all $\vertex\in\vertices_\edge$,
which are precisely the coefficients in \eqref{def:xedgeVAG}.
%Hence \eqref{bc:extv} is written in this case, for $v\in X_{\bcdisc,0}$,
%\be\label{def:vedgeVAG}
%V_\edge = \frac 1 {{\rm Card}(\vertices_\edge)} \sum_{\vertex\in\vertices_\edge} v_\vertex.
%\ee

\item The VAG scheme is the gradient discretisation $\disc$ obtained from  the gradient discretisation $\obcdisc$ by performing a mass-lumping in the sense of Definition \ref{def:ml}. 
We split each tetrahedron $T_{K,\edge,\vertex,\vertex'}$ into
three parts $T_{K,\edge,\vertex,\vertex'}^{K}$, $T_{K,\edge,\vertex,\vertex'}^{\vertex}$,
and $T_{K,\edge,\vertex,\vertex'}^{\vertex'}$ (whose detailed geometry is not needed), and we let $V_K$ be the union of all $(T_{K,\edge,\vertex,\vertex'}^{K})_{\edge,\vertex,\vertex'}$ and $V_\vertex$ be the union of all $(T_{K,\edge,\vertex,\vertex'}^{\vertex})_{K,\edge,\vertex'}$. 
This leads to
\[
\forall v\in X_{\disc,0}\,:\,
\Pi_\disc v=\sum_{K\in\mesh} v_K\chi_{V_K}+\sum_{\vertex\in\mathcal V}
v_\vertex\chi_{V_\vertex}.
\]

\end{enumerate}

\begin{figure}
\begin{center}
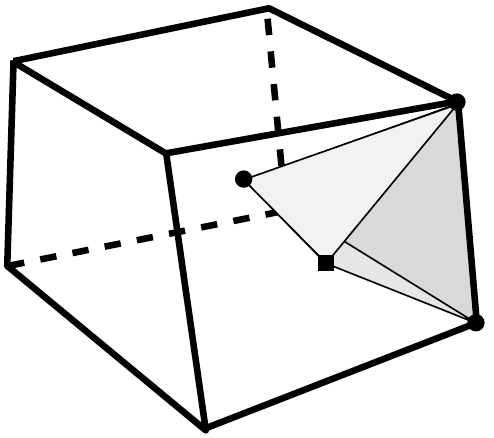
\label{fig:vag}\caption{Definition of tetrahedron $T_{K,\edge,\vertex,\vertex'}$
in a mesh cell $K$.}
\end{center}
\end{figure}

The regularity of a sequence of VAG gradient discretisations
is defined as the regularity of the underlying tetrahedral meshes $\polyd^T$
in the sense of Definition \ref{def:regpolyd}. We can check that $\regbc(\obcdisc)$ remains bounded by a non-decreasing function of $\theta_{\polyd^T}$, and the proof of the property $(\mathcal{P})$ for VAG gradient discretisation
is thus a direct consequence of the results in Section \ref{sec:prop},
especially Theorem \ref{th:bcgd} (properties of the barycentric condensation)
and Theorem \ref{th:ml} (properties of mass-lumped gradient discretisations).

\begin{commentforus}
\begin{remark} An upper bound on $\theta_{\polyd^T}$ gives a control on
$\regbc(\disc)$. More precisely, by writing $a\lesssim b$
as a shorthand for $a\le Cb$ for some $C>0$ only depending
on an upper bound on $\theta_{\polyd^T}$, and $a\sim b$
as a shorthand for $a\lesssim b$ and $b\lesssim a$, we have:
\begin{enumerate}
\item\label{com:it1} \emph{The length of any edge in any tetrahedron $T$
is comparable to the diameter $h_T$ of the tetrahedron} (this one is classical,
but let's write it for sake of completeness): if $T$ is a tetrahedra
of $\mesh^T$, let $\tau$ be a face of $T$ and $\vertex$ be the opposite
vertex. Let $\x_T$ be the centre of the largest ball included in $T$;
this ball has diameter $\sim h_T$.
Let $(\vertex_i)_{i=1,\ldots,d}$ be the vertices of $\tau$. We
write $\x_T$ as a convex combination $\x_T=\lambda\vertex + \sum_{i=1}^d
\lambda_i\vertex_i$ and, for any $\y$ in the hyperplane spanned by $\tau$, since $(\y-\vertex_i)\bot\n_{T,\tau}$
and $\lambda+\sum_{i=1}^d\lambda_i=1$ we have
\[
(\x_T-\y)\cdot\n_{T,\tau}=\lambda(\vertex-\y)\cdot\n_{T,\tau}+\sum_{i=1}^d
\lambda_i(\vertex_i-\y)\cdot\n_{T,\tau}=
\lambda(\vertex-\y)\cdot\n_{T,\tau}.
\]
Hence $(\x_T-\y)\cdot\n_{T,\tau}\le \lambda|\vertex-\y|\le |\vertex-\y|$.
But $(\x_T-\y)\cdot\n_{T,\tau}={\rm dist}(\x_T,\tau)\ge \rho_T\sim h_T$,
and therefore 
\be\label{comp0}
\begin{array}{c}
\mbox{For all $\y$ in the hyperplane generated by a face $\tau$ of a tetraedron $T$,}\\
\mbox{if $\vertex$ is the vertex opposite to $\tau$ in $T$ then $|\vertex-\y|\sim h_T$}.
\end{array}
\ee
Taking $\y$ as the vertices of $\tau$ proves that
\be\label{comp1}
\mbox{The length of any edge in a tetrahedron $T$ is $\sim h_T$.}
\ee

\item \emph{If $\edge$ is a face and $h_\edge$ the maximal distance between two of its
vertices, then $h_\edge$ is similar to the diameter of any tetrahedron with its
base on the face ($\edge$ does not need to be planar)}:
we first notice that any two tetrahedra $T$ and $T'$ with base on $\edge$ share 
the common $[\x_K,\centeredge]$ and thus, by \eqref{comp1}, 
\be\label{comp2}
h_{T}\sim |\x_K-\centeredge|\sim h_{T'}.
\ee
We have $h_\edge=|\vertex_1-\vertex_2|$ for some vertices $\vertex_{i}$
of $\edge$. Let us take $T_1$ and $T_2$ tetrahedra
in $K$ with base on $\edge$ and having respectively $\vertex_1$ and $\vertex_2$
as vertices. Using \eqref{comp2} we have 
\be\label{comp33}
h_\edge=|\vertex_1-\vertex_2|\le |\vertex_1-\centeredge|+|\centeredge-\vertex_2|
\le h_{T_1}+h_{T_2}\sim h_T
\ee
for any tetrahedron $T$ with base on $\edge$.
Any edge of $\edge$ is also an edge of a tetrahedron with base on $\edge$;
Properties \eqref{comp1} and \eqref{comp33} therefore give
\be\label{comp4}
\mbox{The length of any edge of $\edge$ is $\sim h_\sigma$}.
\ee
Each tetrahedron $T$ with its base on $\edge$ share an edge with $\edge$.
Hence, \eqref{comp1} and \eqref{comp4} show that
\be\label{comp5}
\mbox{For any tetrahedron $T$ with its base on $\edge$,
$h_T\sim h_\sigma$}.
\ee
\item \emph{Upper bound on $\regbc(\disc)$}: for the VAG method, $U\in\partition$
is a tetrahedron $T$ with its base on $\edge$, and the only degree of freedom that is
eliminated in $I_U$ (from the $\P_1$ method) is the degree of freedom at $\centeredge$.
This elimination is done by using the vertices of $\edge$ and, by \eqref{comp5},
these vertices all lie within distance $h_\sigma\sim h_T={\rm diam}(U)$ of the points in $U$.

\end{enumerate}
Note that if we also control the number of faces in each cell of $\mesh$, then by working
neighbour to neighbour we can show that all tetrahedra have a size $h_T\sim h_K$.
\end{remark}
\end{commentforus}

\section{Conclusion}\label{sec:ccl}

We gave here a brief presentation of gradient schemes, a generic framework for the convergence
analysis of several numerical methods for various diffusion models. 
This framework is based
on the notion of gradient discretisations (a triplet of discrete space, reconstructed gradient
and reconstructed function) and on core properties they must satisfy to ensure the convergence
of the corresponding gradient schemes. We provided generic tools to prove that given numerical schemes
can be associated with gradient discretisations that satisfy the core properties: local linearly
exact gradients, barycentric condensation, mass lumping and control by polytopal toolboxes.
We then showed that several classical methods are gradient schemes: conforming $\P_k$ finite elements
(and mass-lumped $\P_1$ finite elements), non-conforming $\P_1$ finite elements
(with or without mass-lumping),
$\RTk$ mixed finite elements, multi-point flux approximation O-scheme,
discrete duality finite volumes, hybrid mimetic mixed methods (which contains hybrid
mimetic finite differences), nodal mimetic finite differences, and the vertex approximate
gradient scheme. All these schemes  have  been shown to be associated with
gradient discretisations that satisfy the required core properties. 

Ongoing works concern the adaptation of the gradient scheme framework to some more general operators. 
In the case of the incompressible Stokes equations \cite{dro-14-sto}, it has been possible to obtain convergence results which simultaneously hold for the Taylor-Hood scheme, the Crouzeix-Raviart scheme and the MAC scheme. 
Results in this direction have also been obtained on the elasticity problems \cite{dro2015graelas}. 
Some interesting questions remain open even in the case of the Laplace equation. 
For example, it is still not known whether discontinuous Galerkin schemes fall into the gradient scheme framework, {\it i.e.} if it is possible to construct a gradient that gathers the consistent part of the discontinuous Galerkin gradient and the jumps penalisation. 
As shown in the study of HMM methods\cite[Eq. (5.11)]{dro-12-gra}, such a construction would require some form of orthogonality property between these two components of discontinuous Galerkin schemes; further investigation is necessary.

\begin{acknowledgement}
\textbf{Acknowledgement}: the authors would like to thank Boris Andreianov for insightful discussions during the course of this research.
\end{acknowledgement}

\bibliographystyle{abbrv}
\bibliography{gabib}

\end{document}